\documentclass{amsart}

\usepackage{amsmath,amsthm,amssymb,amsfonts,verbatim}
\usepackage{mathtools}

\usepackage[usenames,dvipsnames]{xcolor}
\usepackage{mathrsfs}
\usepackage{upgreek}
\usepackage{bbm}
\usepackage{accents}

\usepackage{tikz-cd,extarrows}
\usepackage{rotating}
\usepackage{todonotes}

\usepackage[left=3.5cm,right=3.5cm,top=3.5cm,bottom=3.5cm]{geometry}

\usepackage{microtype}
\usepackage{verbatim}
\usepackage{comment}
\usepackage{enumitem,kantlipsum}
\usepackage{fouridx}
\usepackage{cite}
\usepackage{hyperref}
\usepackage{url} 

\numberwithin{equation}{section}

\newcommand{\E}{\mathbb{E}}
\newcommand{\A}{\mathcal{A}}

\newcommand{\rodd}{[r]_{\text{odd}}}
\newcommand{\ba}{\pmb{\alpha}}
\newcommand{\bb}{\pmb{\beta}}
\newcommand{\msympn}{P^{\Sym}_{n,p}}
\newcommand{\msympinf}{P^{\Sym}_{\infty,p}}
\newcommand{\mktilda}{\tilde{P}^{\Sym,\kappa}_{p}}
\newcommand{\mkNormal}{P^{\Sym,\kappa}_{p}}
\newcommand{\mkStilda}{\tilde{P}^{\Sym,\kappa}_{p,S}}

\newcommand{\mkAtilda}{\tilde{P}^{\Sym,\kappa}_{p,\ba}}
\newcommand{\mkANormal}{P^{\Sym,\kappa}_{p,\ba}}
\newcommand{\mkBtilda}{\tilde{P}^{\Sym,\kappa}_{p,\bb}}
\newcommand{\spmk}{P^{\Sym,(k)}_p}
\newcommand{\qpmk}{P_{p,\alpha}^{\textup{M},(k)}}
\newcommand{\mMalle}{P_{\infty;1,t}}
\newcommand{\mGarton}{P_{\infty;n,0}}
\newcommand{\mMalled}{P_{d;1,t}}
\newcommand{\mGartond}{P_{d;n,0}}

\renewcommand{\P}{\mathbb{P}}
\renewcommand{\leq}{\leqslant}
\renewcommand{\geq}{\geqslant}

\newcommand{\Z}{\mathbb{Z}}
\newcommand{\R}{\mathbb{R}}

\DeclareMathOperator{\Mat}{Mat}
\DeclareMathOperator{\Hom}{Hom}
\DeclareMathOperator{\cok}{cok}
\DeclareMathOperator{\Sym}{Sym}
\DeclareMathOperator{\Aut}{Aut}
\DeclareMathOperator{\Sur}{Sur}
\DeclareMathOperator{\im}{Im}
\DeclareMathOperator{\rank}{rank}

\newcommand{\ceil}[1]{\lceil#1\rceil}
\newcommand{\floor}[1]{\lfloor#1\rfloor}
\newcommand{\qbin}[2]{\genfrac{[}{]}{0pt}{}{#1}{#2}}
\newcommand{\abs}[1]{\vert#1\vert}
\newcommand{\la}{\lambda}
\newcommand{\qhyp}[5]{\fourIdx{}{#1}{}{#2}\phi\bigg[\genfrac{}{}{0pt}{}{#3}{#4};#5\bigg]}

\allowdisplaybreaks[3]

\newtheorem{thm}{Theorem}[section]
\newtheorem{conj}[thm]{Conjecture}
\newtheorem{prop}[thm]{Proposition}
\newtheorem{cor}[thm]{Corollary}
\newtheorem{lemma}[thm]{Lemma}

\begin{document}

\title[Sandpile Groups of Random Bipartite Graphs]{Sandpile groups of random bipartite graphs and families of distributions with the same moments}

\author{Jason Fulman}
\address{Department of Mathematics,
University of Southern California}
\email{fulman@usc.edu}

\author{Nathan Kaplan}
\address{Department of Mathematics, University of California, Irvine}
\email{nckaplan@math.uci.edu}

\author{Deepesh Singhal}
\address{Department of Mathematics, University of California, Irvine}
\email{singhald@uci.edu}

\author{S. Ole Warnaar}
\address{School of Mathematics and Physics,
The University of Queensland}
\email{o.warnaar@maths.uq.edu.au}

\begin{abstract}
Recently, there has been significant interest in applying the method of moments developed by Wood and others to study distributions of finite abelian groups that arise in number theory and combinatorics.  When the moments do not grow too fast, they determine a unique distribution.  We construct large families of distributions that have the same moments.  These families include several distributions that arise naturally in the study of sandpile groups of families of random graphs.  Wood determined the distribution of Sylow $p$-subgroups of sandpile groups of Erd\H{o}s--R\'enyi random graphs.  This was extended by M\'esz\'aros to sandpile groups of random $d$-regular graphs, who observed an interesting special case when $d$ is even and $p = 2$. We study Sylow $p$-subgroups of sandpile groups of random bipartite graphs and similarly find a special case for $p =2$.  Although this distribution differs from that of M\'esz\'aros, we show that they have the same moments and fit into our broader construction.  To compute the moments of the distributions we study, we apply combinatorial tools from the theory of Hall--Littlewood functions.

\end{abstract}

\maketitle

\tableofcontents

\section{Introduction}

A major goal of this paper is to study distributions of finite abelian $p$-groups that arise in the study of Sylow $p$-subgroups of sandpile groups of bipartite graphs.
We find an interesting special case that occurs when $p=2$ and investigate how it fits into a larger family of distributions.
We begin by recalling some background.

\subsection{The sandpile group of a graph}

For $\Gamma$ a graph with vertices labeled $1,2,\dots,n$, let $\deg(i,j)$ denote the number of edges between the vertices labeled $i$ and $j$, where we allow for $\deg(i,j)>1$.
The \emph{Laplacian} of $\Gamma$, denoted $L(\Gamma)$ or just $L$, is the $n\times n$ matrix with entries
\[
L_{ij}=\begin{cases} 
-\deg(i,j) & \text{for $i\neq j$}, \\
\deg(i)& \text{for $i=j$}.
\end{cases}
\]
The matrix $L$ defines a linear map $L \colon \Z^n \to \Z^n$.
Let $\Z_0^n$ denote the subspace of $\Z^n$ consisting of vectors whose elements sum to $0$.
Since each row and column of $L$ sums to $0$, the image of $L$ lies inside $\Z_0^n$.
The cokernel of $L$, denoted $\cok(L)$, is 
\[
\cok(L) = \Z^n/\im(L) \cong \Z \oplus \Z_0^n/\im(L).
\]
It is not difficult to show that the \emph{corank} of $L$, that is, $n-\rank(L)$, is equal to the number of connected components of $\Gamma$.  
In the case that $\Gamma$ is connected, $\Z_0^n/\im(L)=S_{\Gamma}$ is a finite abelian group called the \emph{sandpile group} of $\Gamma$.
This group, which contains key information about $\Gamma$, is the subject of much current research.
For example, Kirchhoff's matrix tree theorem shows that $\abs{S_{\Gamma}}$ is equal to the number of spanning trees of $\Gamma$.
We note that for a connected graph $\Gamma$, $S_{\Gamma}$ can also be defined as the cokernel of the \emph{reduced Laplacian} of $\Gamma$, which is the matrix obtained from $L(\Gamma)$ by deleting its $i$\textsuperscript{th} row and column, for any $i$.

For a fixed constant $0<u<1$, let $G(n,u)$ denote an Erd\H{o}s--R\'enyi random graph on $n$ vertices such that each of the $\binom{n}{2}$ potential edges is included independently with probability $u$.
For fixed $u$, the limit as $n\to\infty$ of the probability that $G(n,u)$ is connected is $1$.
If $S_{G(n,u)}$ denotes the sandpile group of an Erd\H{o}s--R\'enyi random graph $G(n,u)$, this implies that the limit as $n\to\infty$ of the probability that $S_{G(n,u)}$ is a finite abelian group is $1$.
If $G$ is a finite abelian group, then we write $G_p$ for the Sylow $p$-subgroup of $G$.

\subsection{Distributions of partitions and finite abelian $p$-groups}

A partition $\la=(\la_1,\la_2,\dots)$ of $n$ is a weakly decreasing sequence of nonnegative integers such that $\abs{\la}:=\la_1+\la_2+\cdots=n$. 
The positive $\la_i$ are called the parts of $\la$ and the number of parts, referred to as the length of the partition, is denoted by $l(\la)$.
We adopt the standard conventions of writing the unique partition of $0$ as $0$ and the `rectangular' partition consisting of $r$ parts of size $k$ as $(k^r)$.
Given a partition $\la$, we denote by $\la'=(\la'_1,\la'_2,\dots)$ its conjugate, so that $\la'_i$ is the number of parts of $\la$ of size at least $i$.
In particular $\la_1' = l(\la)$. 
Let $m_i(\la):=\la'_i-\la'_{i+1}$ be the multiplicity of parts of size $i$, i.e., the number of parts of $\la$ of size exactly $i$.
If all parts of $\la$ are distinct, so that $m_i(\la)\leq 1$ for all $i\geq 1$, then we say that $\la$ is a strict partition.
For partitions $\la,\mu$ such that $\la_i\geq\mu_i$ for all $i\geq 1$, we write $\mu\subseteq\la$ and say that $\mu$ is contained in $\la$.
For $\mu$ contained in $\la$ we by abuse of notation write
\[
n(\la/\mu):=\sum_{i\geq 1}\binom{\la'_i-\mu'_i}{2},
\]
where $n(\la/0)$ is more simply written as $n(\la)$.
Here we note that $n(\la)$ may also be written as $n(\la)=\sum_{i\geq 1} (i-1)\la_i$, in accordance with its usual definition, see \cite[page~3]{macdonald1995symmetric}.

We define the $q$-shifted factorials
\[
(x;q)_{\infty}:=\prod_{i=1}^{\infty}(1-xq^{i-1})
\]
and, for $n$ an arbitrary integer,
\[
(x;q)_n:=\frac{(x;q)_{\infty}}{(xq^n;q)_{\infty}}.
\]
Thus, $(x;q)_0=1$, 
\[
(x;q)_n=\prod_{i=1}^n(1-xq^{i-1})
\]
for $n$ a positive integer, and $1/(q;q)_n=0$ for $n$ a negative integer.
For arbitrary integers $m,n$ we define the $q$-binomial coefficient
\begin{equation*}
\qbin{m}{n}_q 
:=\begin{cases}\displaystyle \frac{(q;q)_m}{(q;q)_n (q;q)_{m-n}}
& \text{if $0\leq n\le m$}, \\[8pt]
0 & \text{otherwise}. \end{cases}
\end{equation*}

Let $p$ be a prime and $\la$ a partition of length $l(\la)=r$.
Then a finite abelian $p$-group of \emph{type} $\la$ and \emph{rank} $r$ is a group of the form
\[
G_{\la}:=\bigoplus_{i=1}^r \Z/p^{\la_i}\Z.
\]
The \emph{exterior square} of the abelian group $G$ is the group
\[
\wedge^2 G := (G\otimes G)/ \langle g\otimes g:\; g\in G \rangle.
\]
We recall from \cite[Section 2.4]{wood2017distribution} that
\[
\wedge^2 G_{\la} \cong \bigoplus_{i\geq 1} \big(\Z/p^{\la_i} \Z\big)^{i-1}.
\]
In particular, $\abs{\wedge^2 G_{\la}} = p^{n(\la)}$.

In order to simplify notation, for the remainder of the paper, if $p>1$ is a real number then $q:=p^{-1}\in(0,1)$.

For $R$ a ring, let $\Mat_n(R)$ denote the set of $n\times n$ matrices with entries in $R$ and
$\Sym_n(R)\subseteq\Mat_n(R)$ the set of symmetric such matrices. 
For $p$ a prime, let $\Z_p$ be the ring of $p$-adic integers and $A\in\Sym_n(\Z_p)$ a random matrix chosen with respect to additive Haar measure on $\Sym_n(\Z_p)$.
It follows from \cite[Theorem 2]{clancy2015cohen} that
\[
\P_{A\in\Sym_n(\Z_p)}\big(\cok(A)\cong G_{\la}\big)=\msympn(\la),
\]
where
\[
\msympn(\la):=\frac{q^{n(\la)+\abs{\la}}(q;q)_n (q;q^2)_{\ceil{(n-l(\la))/2}}}{(q;q)_{n-l(\la)} \prod_{i\geq 1} (q^2;q^2)_{\floor{m_i(\la)/2}}}.
\]
Note that $\msympn(\la)=0$ if $l(\la)>n$.
Letting $n$ tend to infinity gives a distribution on partitions where a partition $\la$ is chosen with probability
\begin{equation}\label{symforminf}
\msympinf(\la):=\frac{q^{n(\la)+\abs{\la}}(q;q^2)_{\infty}}{\prod_{i\geq 1}(q^2;q^2)_{\floor{m_i(\la)/2}}}.
\end{equation}
Taking limits, it also follows from \cite[Theorem 2]{clancy2015cohen} that
\[
\lim_{n\to\infty}\P_{A\in \Sym_n(\Z_p)}\big(\cok(A)\cong G_{\la}\big)=\msympinf(\la).
\]
Although the distribution $\msympinf$ was originally defined for prime $p$, it can be defined for any real $p>1$ by \eqref{symforminf}, see~\cite[Section 4]{fulman2019random}.

\subsection{Sylow $p$-subgroups of sandpile groups of random graphs}

Wood proved a strong universality result for cokernels of random symmetric $p$-adic matrices \cite[Theorem 1.3]{wood2017distribution}.
Applying this result shows that even though the reduced Laplacian of an Erd\H{o}s--R\'enyi random graph does not give a Haar random element of $\Sym_n(\Z_p)$, as $n\to\infty$, the distribution of Sylow $p$-subgroups of the sandpile group of an Erd\H{o}s--R\'enyi random graph with $n$ vertices converges to $\msympinf$.

\begin{thm}[{\!\!\cite[Theorem 1.1]{wood2017distribution}}]\label{thm:wood_p-part}
Let $p$ be a prime and $\la$ a partition.
Then for an Erd\H{o}s--R\'enyi random graph $G(n,u)$,
\[
\lim_{n\to\infty}\P\big((S_{G(n,u)})_p\cong G_{\la}\big)=\msympinf(\la).
\]
\end{thm}

Nguyen and Wood proved stronger results about sandpile groups of random graphs \cite{NguyenWood}.
For example, they determined the probability that $S_{G(n,u)}$ is cyclic.
The sandpile group of a connected graph comes with a canonical duality pairing.
In \cite[Theorem 1.1]{Hodges} Hodges proved an analogue of Theorem~\ref{thm:wood_p-part} that determines the probability that a finite abelian $p$-group arises as $(S_{G(n,u)})_p$ together with a particular choice of pairing.
In this paper, we do not pursue questions related to pairings on sandpile groups, nor do we consider the joint distribution of $(S_{G(n,u)})_p$ at multiple primes.

Now that the distribution as $n\to\infty$ of Sylow $p$-subgroups of Erd\H{o}s--R\'enyi random graphs is understood, it is natural to ask about the distribution of Sylow $p$-subgroups of sandpile groups for other classes of random graphs.
M\'esz\'aros determined the distribution of Sylow $p$-subgroups of sandpile groups of random $d$-regular graphs on $n$ vertices, as $n\to\infty$ \cite{meszaros2020distribution}.
Suppose $d\geq 3$ and $n$ is even.
The random $d$-regular graph $H_n$ is obtained by taking the union of $d$ independent uniform random perfect matchings on the set of $n$ vertices.

\begin{thm}[{\!\!\cite[Theorem 1.2]{meszaros2020distribution}}]\label{mes_thm}
Let $p$ be a prime and $\la$ a partition.
\begin{enumerate}
\item If $d$ is odd, or $d$ is even and $p$ is odd, then
\[
\lim_{n\to\infty} \P\big((S_{H_n})_p \cong G_{\la}\big) 
=\msympinf(\la).
\]
\item If $d$ is even, then $(S_{H_n})_2$ always has odd rank.
Moreover, in this case
\[
\lim_{n\to\infty} \P\big((S_{H_n})_2 \cong G_{\la}\big) 
= \begin{cases}
 2^{l(\la)} P_{\infty,2}^{\Sym}(\la) & \text{if $l(\la)$  is odd},\\
 0 & \text{otherwise}.
 \end{cases}
\]
\end{enumerate}
\end{thm}

\subsection{Sylow $p$-subgroups of sandpile groups of random bipartite graphs}

We focus on Sylow $p$-subgroups of sandpile groups of random bipartite graphs.
Let $n_1\geq n_2$ be positive integers and $u$ a fixed real number $0<u<1$.
The Erd\H{o}s--R\'enyi random bipartite graph $G(n_1,n_2,u)$ is a bipartite graph with vertex sets $V_1$ of size $n_1$ and $V_2$ of size $n_2$, where each of the $n_1 n_2$ potential edges between a vertex in $V_1$ and a vertex in $V_2$ is included independently with probability $u$.
Fixing a constant $0<\alpha\leq1$, we are interested in the large $n$ limit of
\[
G_{\alpha}(n,u):=G(n,\ceil{\alpha n},u).
\] 
Let $S_{G_{\alpha}(n,u)}$ denote the sandpile group of the graph $G_{\alpha}(n,u)$. 
We pose the following conjecture.

\begin{conj}\label{p-Syl_conj}
Let $p$ be prime, $\alpha,u$ real numbers such that $\frac{1}{p}<\alpha\leq 1$ and $0<u<1$, and $\la$ a partition.
Then,
\[
\lim_{n\to\infty} \P\big((S_{G_{\alpha}(n,u)})_p \cong G_{\la}\big) =
\begin{cases} \msympinf(\la) & \text{if $p$ is odd}, \\[1mm]
2^{l(\la)-1} P_{\infty,2}^{\Sym}(\la) & \text{if $p=2$}.
\end{cases}
\]
\end{conj}

In Section~\ref{sec_data} we present computational evidence in support of this conjecture. 
A result of Koplewitz implies that the condition $\alpha>\frac{1}{p}$ is necessary.

\begin{thm}[{\!\!\cite[Theorem~1]{koplewitz2023sandpile}}]
Let $p$ be a prime, $0<\alpha<1$ and $0<u<1$.
Then
\[
\lim_{n\to\infty}\E\big(\rank\big((S_{G_{\alpha}(n,u)})_p\big)\big)
=\begin{cases}
O(1) &\text{if $\frac{1}{p}<\alpha<1$},\\[4pt]
\sqrt{\frac{(p-1)n}{2\pi p^2}} +O(1) &\text{if $\alpha=\frac{1}{p}$},\\[4pt]
\big(\frac{1}{p}-\alpha\big)n +O(1) &\text{if $\alpha<\frac{1}{p}$}.
\end{cases}
\]
\end{thm}

In fact, his arguments show that when $\alpha<\frac{1}{p}$, for any $G$,
\[
\lim_{n\to\infty} \P\big((S_{G_{\alpha}(n,u)})_p \cong G\big)=0.
\]
Conjecture~\ref{p-Syl_conj} is related to the work of Bhargava, dePascale and Koenig on the distribution of the ranks of sandpile groups of random directed bipartite graphs \cite{bhargava2023rank}.
In forthcoming work \cite{Singhal_sandpile} by the third author, a proof of Conjecture~\ref{p-Syl_conj} for odd primes $p$ will be presented.

\subsection{Distributions of partitions with the same moments}

A major focus of this paper is the two distributions corresponding to
Theorem~\ref{mes_thm} for even $d$ and $p=2$, and Conjecture~\ref{p-Syl_conj} for $p=2$.
These two distributions are not the same, but are closely related, and
we shall explain how they fit into a family of distributions.
Before defining this family, we recall some material concerning moments of distributions of partitions.

Let $p$ be a prime and $G_{\la}$ and $G_{\mu}$ two finite abelian $p$-groups, indexed by the partitions $\la$ and $\mu$, respectively.
Then we write $\Sur_p(\la,\mu)$ for the set of surjective group homomorphisms from $G_{\la}$ to $G_{\mu}$ and set $\Aut_p(\la)=\Sur_p(\la,\la)$.
Nguyen and Van Peski have shown that $\abs{\Sur_p(\la,\mu)}$ can be expressed in terms of Hall--Littlewood symmetric functions.
For the definitions of the Hall--Littlewood symmetric functions $P_{\la/\mu}$, $P_{\la}$, and $Q_{\mu}$ we refer the reader to Section~\ref{Subsec: defn of sym polynomial}.

\begin{prop}[{\!\!\cite[Proposition 6.2]{nguyen2024universality}}]\label{surjhall}
For partitions $\la,\mu$ and $p$ a prime,
\begin{equation}\label{Surp-size}
\abs{\Sur_p(\la,\mu)}
= \frac{P_{\la/\mu}(q,q^2,\dots;q)}{P_{\la}(q,q^2,\dots;q) Q_{\mu}(1,q,q^2,\dots;q)},
\end{equation}
where we recall that $q:=1/p$.
\end{prop}

When $\mu\not\subseteq\la$, then $P_{\la/\mu}$ is zero, so $\abs{\Sur_p(\la,\mu)}=0$ as expected.
Although $\Sur_p(\la,\mu)$ is not defined when $p$ is not a prime, \eqref{Surp-size} may be used to define $\abs{\Sur_p(\la,\mu)}$ for all $p>1$.
Therefore, given a real number $p>1$, a measure $\nu$ on the set of partitions, and a partition $\mu$, we can define the \emph{$\mu$-moment} of $\nu$ by
\begin{equation}\label{mumoment}
M_{p,\mu}(\nu):=\sum_{\la}\nu(\la)\abs{\Sur_p(\la,\mu)}.
\end{equation}
When $p$ is prime, this is the expected number of surjections from a random group chosen from $\nu$ to the fixed finite abelian $p$-group $G_{\mu}$.

Clancy, Kaplan, Leake, Payne, and Wood computed the moments of the distribution $\msympinf$.

\begin{thm}[{\!\!\cite[Theorem 11]{clancy2015cohen}}]\label{Thm: moment P sym inf}
Let $p>1$ be a real number and $\mu$ a partition.
Then the $\mu$-moment of $\msympinf$ is
\[
M_{p,\mu}\big(\msympinf\big)=p^{n(\mu)}.
\]
Moreover, when $p$ is prime, this can also be expressed as $\abs{\wedge^2 G_{\mu}}$.
\end{thm}

Letting $d$ and $n$ go to infinity in Theorem~\ref{thm: moments of Garton} below gives a new proof of Theorem~\ref{Thm: moment P sym inf}, valid for all real $p>1$.
Wood showed that two distinct measures cannot have the same moments provided the moments are not too large.

\begin{thm}[{\!\!\cite[Theorem 8.2]{wood2017distribution}}]\label{Thm: Wood universality}
Let $p>1$ be a real number and $\nu_1,\nu_2$ a pair of measures on the set of partitions.
Suppose that for every $\mu$,
\[
M_{p,\mu}(\nu_1)=M_{p,\mu}(\nu_2)\leq p^{n(\mu)}.
\]
Then $\nu_1=\nu_2$.
\end{thm}

The theorem implies that $\msympinf$ is the only measure whose $\mu$-moment is $p^{n(\mu)}$ for every $\mu$.
However, moments no longer determine the measure for larger values of the moments.
A first example is given by the distribution $\nu$ that corresponds to Theorem~\ref{mes_thm} for even $d$ and $p=2$.
For this distribution, $M_{p,\mu}(\nu) = 2^{n(\mu)+l(\mu)}$, see \cite[Lemmas 8.11 \& 8.12]{meszaros2020distribution}, and so we cannot apply Theorem~\ref{Thm: Wood universality}.
In fact, there is a one-parameter family of distributions sharing the same moments that includes the $p=2$ special case of Conjecture~\ref{p-Syl_conj}.
We generalize this result by constructing large families of distributions that share the same moments.
These are interesting to consider in light of recent advances in the study of the method of moments.
For an overview of these developments, see the ICM notes of Wood \cite{wood2022probability}, and work of Sawin and Wood in which they consider the question of when a collection of moments determines a unique distribution in a very general setting~\cite{SawinWood}.

We construct families of measures such that all measures within the same family have the same moments.
For $\kappa$ a partition, define the measure
\begin{equation}\label{mktilda}
\mktilda(\la):=p^{\sum_{i\geq 1} \kappa_i\la'_i} \msympinf(\la).
\end{equation}
Let $[r]:=\{1,\dots,r\}$.
Then for $\kappa$ a partition of length $r$ and $S\subseteq [r]$, we further define
\begin{equation}\label{signed}
\mkStilda(\la):=\Big(1+(-1)^{\sum_{i\in S}\la'_i} \Big)\mktilda(\la).
\end{equation}

\begin{thm}\label{Thm: Same moments S odd size}
Let $p>1$ be a real number and $\mu,\kappa$ partitions.
Then there exists a constant $C=C_{\kappa}$, independent of $\mu$, such that
\[
M_{p,\mu}\big(\mktilda\big) \leq C p^{n(\mu)+\abs{\kappa}\,l(\mu)}.
\]
Moreover, if $\kappa$ is strict and $r:=l(\kappa)$, then for every subset $S\subseteq [r]$ of odd size,
\begin{equation}\label{equalmeasure}
M_{p,\mu}\big(\mkStilda\big)=M_{p,\mu}\big(\mktilda\big).
\end{equation}
\end{thm}

The second part of the theorem implies that by fixing a strict partition of length $r$, we obtain $2^{r-1}+1$ measures all of which have the same moments.
We can take affine combinations of these measures to get a $2^{r-1}$-parameter family of measures that all have the same moments, where care must be taken to ensure that no partition is chosen with negative probability.

Let $\rodd$ be the set of all subsets of $[r]$ of odd size and $\ba:\rodd\to\mathbb{R}$.
Then $\A_r\subseteq\R^{\rodd}$ is defined as the polytope determined by the $2^r$ linear inequalities
\[
\sum_{S\in\rodd}
(-1)^{\abs{S\cap T}} \ba(S) \leq 1
\]
for each $T\subseteq [r]$.
The ambient dimension of $\A_r$ is $\abs{\rodd}=2^{r-1}$. 
Moreover, $\A_r$ is full-dimensional in $\R^{2^{r-1}}$ since
\[
\Big\{\ba\in\R^{\rodd}: \sum \abs{\ba(S)}\leq 1\Big\}\subseteq \A_r.
\]
We will show in Lemma~\ref{lem:Am-is-cube} that $\A_r$ is a $2^{r-1}$-dimensional cube.
Given $\ba\in\A_r$, we define the measure
\begin{equation}\label{mkAtilda}
\mkAtilda=\sum_{S\in \rodd} \ba(S)\mkStilda
+\bigg(1-\sum_{S\in \rodd} \ba(S)\bigg) \mktilda.
\end{equation}
Note that if $\ba(S)=0$ for all $S\in\rodd$, then $\mkAtilda$ simplifies to $\mktilda$.
The linear inequalities that define $\A_r$ are chosen so that $\mkAtilda$ does not choose any partition with a negative probability.

\begin{thm}\label{Thm: Same moments alpha family}
Let $p>1$ be a real number and $\mu,\kappa$ partitions such that $\kappa$ is strict and $l(\kappa)=r$.
Then 
$\mkAtilda$ is a measure for all $\ba\in\A_r$ \textup{(}i.e.,
$\mkAtilda(\la)\geq 0$ for all partitions $\la$\textup{)} such that
\begin{equation}\label{noalphadependence}
M_{p,\mu}\big(\mkAtilda\big)=M_{p,\mu}(\mktilda).
\end{equation}
Furthermore, if $\mkAtilda=\mkBtilda$ for $\ba,\bb\in \A_r$, then $\ba=\bb$.
\end{thm}

In particular, if $\kappa=(r,r-1,\dots,1)$ is the staircase partition of length $r$, then we have a $2^{r-1}$-parameter family of measures parametrized by $\ba\in \A_r$.
For each $\ba$, we have
\[
M_{p,\mu}(\mkAtilda)=M_{p,\mu}(\mktilda)\leq C p^{n(\mu)+\binom{r+1}{2} l(\mu)}.
\]

We highlight the $r=1$ case.
Note that $\A_1=[-1,1]\subseteq \R^1$.
Given $\alpha\in[-1,1]$ and $k$ a positive integer, we define the distributions
\[
\spmk(\la):=\frac{q^{\binom{k}{2}}}{(-1;q)_k}\,\tilde{P}^{\Sym,(k)}_p(\la)
\quad\text{and}\quad
\qpmk:= \frac{q^{\binom{k}{2}}}{(-1;q)_k}\,\tilde{P}^{\Sym,(k)}_{p,\alpha}(\la).
\]
We note that the distribution appearing in the special case of Conjecture~\ref{p-Syl_conj} for $p=2$ is $P^{\Sym,(1)}_2(\la)$. 
The superscript M in $\qpmk$ stands for `M\'esz\'aros'.
This is to highlight the fact that the special case of Theorem~\ref{mes_thm} corresponding to even $d$ and $p=2$ is the $p=2$, $k=1$, $\alpha=-1$ instance of the family $\qpmk$.

\begin{thm}\label{Thm_r=1_one-par-family}
For $\alpha\in[-1,1]$ and $k$ a positive integer,
\[
M_{p,\mu}\big(\qpmk\big)=M_{p,\mu}\big(\spmk\big)=p^{n(\mu)+kl(\mu)}.
\]
\end{thm}

Note that if we take $\mu$ equal to the empty partition, $\mu=0$, then $\sum_{\la} \qpmk(\la)=1$ and $\sum_{\la} \spmk(\la)=1$, meaning that these are distributions. 

It follows from \cite[Theorem 4.1]{fulman2019random} that if $\la$ is chosen according to $\spmk$, then the distribution of $\la'_1 = l(\la)$ is
\[
\P_{\spmk}\big(\la_1'=b\big)
=\frac{1}{(-1;q)_{k}(-q;q)_{\infty}}\,\frac{q^{\binom{b-k+1}{2}}}{(q;q)_b}.
\]
When $k=1$, this same distribution was conjectured as the distribution of ranks of $p$-Selmer groups of elliptic curves by Poonen and Rains in \cite{poonen2012random}.
Moreover, in \cite[Proposition~2.22]{poonen2012random} they showed that for each $\alpha\in [-1,1]$ and $n$ a nonnegative integer,
\[
\E_{P^{\textrm{M},(1)}_{p,\alpha}}\big(p^{n\la_1'}\big)= q^{-\binom{n+1}{2}} (q;q)_n,
\]
which is equivalent to 
\[
M_{p,(1^n)}\big(P^{\textrm{M},(1)}_{p,\alpha}\big)=q^{-\binom{n+1}{2}}.
\]
This is consistent with Theorem~\ref{Thm_r=1_one-par-family}.
Furthermore, in \cite[Corollary~2.12]{wood2022probability} Wood showed that if $\nu$ is a measure such that for each $n\geq 0$, $M_{p,(1^n)}(\nu)=q^{-\binom{n+1}{2}}$, then there exists an $\alpha\in [-1,1]$ such that $\nu$ and $P^{M,(1)}_{p,\alpha}$ have the same distribution of $p$-ranks.
This means that for each $b\geq 0$, 
\[
\P_{\nu}\big(\la_1'=b\big)=\P_{P^{\textrm{M},(1)}_{p,\alpha}}\big(\la_1'=b\big).
\]

Next, we highlight the case $r=2$.
The inequalities defining $\A_2$ give the square
\[
\A_2
=\big\{(\alpha_1,\alpha_2)\in\mathbb{R}^2:\; -1\leq\alpha_1+\alpha_2\leq 1 \text{ and } -1\leq\alpha_1-\alpha_2\leq 1\big\}.
\]
Given a strict partition $\kappa=(\kappa_1,\kappa_2)$ and $(\alpha_1,\alpha_2)\in\A_2$, we have
\begin{align*}
\Tilde{P}^{\Sym,\kappa}_{p,(\alpha_1,\alpha_2)}
&= \alpha_1 \Tilde{P}^{\Sym,\kappa}_{p,\{1\}}
+ \alpha_2 \Tilde{P}^{\Sym,\kappa}_{p,\{2\}}
+(1-\alpha_1-\alpha_2) \Tilde{P}^{\Sym,\kappa}_{p}\\
&=
\Big(1+(-1)^{\la'_1}\alpha_1+(-1)^{\la'_2}\alpha_2\Big) p^{\kappa_1\la_1'+\kappa_2\la_2'} \msympinf(\la).
\end{align*}
Note that the inequalities that define $\A_2$ are precisely the ones needed to ensure that $\Tilde{P}^{\Sym,\kappa}_{p,(\alpha_1,\alpha_2)}$ never takes negative values.

\subsection{Moment computations}\label{sec:compute_moments_intro}

Suppose $\nu$ is a distribution on partitions given in terms of a real number $p>1$ for which $\nu$ is the unique distribution with the collection of moments $M_{p,\mu}(\nu)$. Sawin and Wood described how to go from this collection of moments to a formula for the probability $\nu(\la)$, see~\cite[Theorem 1.6]{SawinWood}.  For another perspective on this question that uses the theory of Macdonald polynomials, see the paper of Van Peski \cite{vanpeski2025}.  
We take the opposite approach here.
If we are given the distribution $\nu(\la)$ for each $\la$, how do we determine the moments?

A useful approach to address this question is to first interpret the distribution in terms of Hall--Littlewood symmetric functions, and to then use properties of these functions to compute moments of the distribution.
Applying this method, we first study the moments of the distribution $\mkAtilda$, proving part of Theorem~\ref{Thm: Same moments alpha family}.
Then, when $r=l(\kappa)=1$, we use these same ideas to prove Theorem~\ref{Thm_r=1_one-par-family}. 
Finally, we apply our methodology to compute the moments of several measures given in terms of a real parameter $p>1$, not necessarily prime.

\subsubsection{Moments of Cohen--Lenstra partitions}

Let $d$ be a nonnegative integer and $u$ a real number such that $0<u<p$.
We consider the distribution $P_{d,u}$ on partitions given by \cite[Equation~(1.3)]{fulman2019random},
\[
P_{d,u}(\la)=\frac{u^{\abs{\la}}}{\abs{\Aut_p(\la)}}\,
(q;q)_{l(\la)}(uq;q)_d \qbin{d}{l(\la)}_q.
\]
This is nonzero only when $l(\la)\leq d$, and simplifies to the distribution of \cite{friedman1989distribution} for $u=1$.
From \cite[p.~181]{macdonald1995symmetric},
\[
\abs{\Aut_p(\la)}=q^{-2n(\la)-\abs{\la}}\, b_{\la}(q),
\]
where $b_{\la}(q)=\prod_{i\geq 1}(q;q)_{m_i(\la)}$.
Therefore,
\[
P_{d,u}(\la)=\frac{u^{\abs{\la}}q^{2n(\la)+\abs{\la}}}{b_{\la}(q)}\,
(q;q)_{l(\la)}(uq;q)_d \qbin{d}{l(\la)}_q.
\]
Since \cite[page 213]{macdonald1995symmetric}
\begin{equation}\label{Eq_p213}
P_{\la}(1,q,\dots,q^{d-1};q)
=\frac{q^{n(\la)}(q;q)_{l(\la)}}{b_{\la}(q)}
\qbin{d}{l(\la)}_{q},
\end{equation}
where $P_{\la}$ denotes a Hall--Littlewood polynomial defined in Section~\ref{Subsec: defn of sym polynomial}, we alternatively have
\begin{equation}\label{nice}
P_{d,u}(\la)=u^{\abs{\la}} q^{n(\la)+\abs{\la}} (uq;q)_d\,
P_{\la}(1,q,\dots,q^{d-1};q),
\end{equation} 
which we consider for arbitrary real $p>1$ and $0<u<p$.
We will give a short, new proof of a closed-form formula for the $\mu$-moment of $P_{d,u}$ first obtained in \cite[Theorem~5.3]{fulman2019random}.
Like the proof of \cite{fulman2019random}, our approach relies on the
theory of Hall--Littlewood symmetric functions.

\begin{thm}[{\!\!\cite[Theorem~5.3]{fulman2019random}}]\label{thm53Fulman}
Let $p,u$ be real numbers such that $p>1$ and $0<u<p$, and let $\mu$ be a partition.
Then the $\mu$-moment of $P_{d,u}$ is given by
\[
M_{p,\mu}(P_{d,u})=u^{\abs{\mu}} (q;q)_{l(\mu)} \qbin{d}{l(\mu)}_q,
\]
which is $0$ for $l(\mu)>d$.
\end{thm}

Equations~\eqref{nice} and \eqref{Eqn: Q la 1 q q2} imply that
\[
P_{\infty,u}(\la):=\lim_{d\to\infty} P_{d,u}(\la)
=(uq;q)_{\infty}\,\frac{u^{\abs{\la}}q^{2n(\la)+\abs{\la}}}{b_{\la}(q)}.
\]
The moments of these distributions are obtained from Theorem~\ref{thm53Fulman} by taking the limit as $d\to\infty$,
\[
M_{p,\mu}(P_{\infty,u})=u^{\abs{\mu}}.
\]
This was previously shown in the remark following Theorem~5.3 of \cite{fulman2019random}.

\subsubsection{Moments of distributions motivated by class groups of number fields in the presence of roots of unity}

Let $K$ be a number field and $p$ a prime.
Suppose that $K$ contains $p^n$-roots of unity but not $p^{n+1}$-roots of unity.
Denote the rank of the unit group of $\mathcal{O}_K$ by $t+1$.
Lipnowski, Sawin, and Tsimerman described a conjectural distribution for the relative class group $\mathrm{Cl}(L/K)$ as $L$ varies among the quadratic extensions of $K$ \cite[Conjecture 1.2]{lipnowski2020cohen}. 
We focus on two special cases of their conjecture.
In \cite{lipnowski2020cohen} two-parameter families of distributions are considered, indexed by nonnegative integers $n$ and $t$. We will consider generalizations of the cases $n=1$ and $t=0$. 

When $n=1$, the conjectured distribution was first described by Malle in~\cite{malle2010distribution} and is given by
\[
\mMalle(\la):=(-q;q)_t (q;q^2)_{\infty}\,
\frac{q^{(t+1)\abs{\la}+2n(\la)-\binom{l(\la)}{2}}
(q^{t+1};q)_{l(\la)}}{b_{\la}(q)}.
\]
This can be generalized to $d\in \Z_{\geq 1}$, $p\in(1,\infty)$ and $t\in(-1,\infty)$ by
\[
\mMalled(\lambda):=P_{d,q^t}(\la)\, 
\frac{q^{-\binom{l(\la)}{2}}(q^{t+1};q)_{l(\la)}}
{(q^{2t+2};q^2)_d}.
\]
Taking the limit $d\to\infty$, we obtain
\begin{equation}\label{Malle}
\mMalle(\la)
:= \lim_{d\to\infty} \mMalled(\la)
=P_{\infty, q^t}(\la)\,
\frac{q^{-\binom{l(\la)}{2}}(q^{t+1};q)_{l(\la)}}
{(q^{2t+2};q^2)_{\infty}}.
\end{equation}
Note that this matches the previous definition if $t$ is a nonnegative integer.
When $p$ is a prime and $t$ a nonnegative integer, the moments were computed in \cite[Section 8.3]{lipnowski2020cohen}.
Here we compute the moments by a different method that works for general $p$, $t$ and also for finite $d$.

\begin{thm}\label{thm: moments of Malle}
Let $p,t$ be real numbers such that $p>1$ and $t>-1$, let $d$ be a positive integer and let $\mu$ be a partition.
Then the $\mu$-moment of the distribution $\mMalled$ is 
\[
M_{p,\mu}\big(\mMalled\big)
=\frac{q^{t\abs{\mu}-\binom{l(\mu)}{2}} (q;q)_{l(\mu)}}{ (-q^{d+t+1-l(\mu)};q)_{l(\mu)} }
\qbin{d}{l(\mu)}_q.
\]
Furthermore, the $\mu$-moment of the distribution $\mMalle$ is 
\[
M_{p,\mu}\big(\mMalle\big)=p^{\binom{l(\mu)}{2}-t\abs{\mu}}.
\]
\end{thm}

When $t=0$ the distributions described in the first paragraph of this section were studied by Garton, who computed their moments \cite[Corollary 3.2.7]{garton2015random}. 
He also gave a closed form formula of the distribution when $n=1$ and $n=2$ in \cite[Theorem 1.2.4]{garton2015random}. 
Lipnowski, Sawin, and Tsimerman \cite[Proposition 8.23]{lipnowski2020cohen} gave a closed form formula for any positive integer $n$ when $t=0$:
\[
\mGarton(\la):=P_{\infty,1}(\la)\,
\frac{q^{-\sum_{i=1}^n\binom{\la'_i}{2}} (q;q)_{\la'_n} }
{(q^2;q^2)_{\infty}} \prod_{i=1}^{n-1} (q;q^2)_{\ceil{m_i(\la)/2}}.
\]
This can be generalized to $d,n\in \Z_{\geq 1}$ and $p\in(1,\infty)$ by
\[
\mGartond(\la):=P_{d,1}(\la)\,
\frac{q^{-\sum_{i=1}^n\binom{\la'_i}{2}}(q;q)_{\la'_n}}{(q;q)_d}\,
(q;q^2)_{\ceil{(d-l(\la))/2}}
\prod_{i=1}^{n-1} (q;q^2)_{\ceil{m_i(\la)/2}}.
\]

\begin{thm}\label{thm: moments of Garton}
Let $n$ and $d$ be positive integers, $p>1$ a real number and $\mu$ a partition.
Then the $\mu$-moment of the distribution $\mGartond$ is 
\[
M_{p,\mu}\big(\mGartond\big)
= q^{-\sum_{i=1}^n\binom{\mu'_i}{2}}\,
\frac{(q;q)_d}{(q;q)_{d-l(\mu)}},
\]
and thus
\[
M_{p,\mu}\big(\mGarton\big)=p^{\sum_{i=1}^n\binom{\mu'_i}{2}}.
\]
\end{thm}

\subsection{Outline of the paper}

In the next section, we present computational results supporting Conjecture~\ref{p-Syl_conj}.
In Section~\ref{sec:Background} we introduce some background material, focusing 
on the combinatorial tools needed to compute moments later in the paper.
In Section~\ref{Sec: computing moments} we use these tools to study the moments of the measures discussed in the introduction, proving Theorems~\ref{Thm: Same moments S odd size}, \ref{Thm: Same moments alpha family} and \ref{Thm_r=1_one-par-family}.
In Section~\ref{Sec: Moments of other measures}, we compute the moments of the distributions introduced in Section~\ref{sec:compute_moments_intro}, thereby reproving Theorem~\ref{thm53Fulman} and proving Theorems~\ref{thm: moments of Malle} and \ref{thm: moments of Garton}.
In the final section, we prove several technical lemmas that are used earlier in the paper.

\section{Numerical Experiments Supporting Conjecture \ref{p-Syl_conj}}\label{sec_data}

In this section, we present numerical experiments providing evidence for Conjecture~\ref{p-Syl_conj}.  
All experiments were carried out in the computer algebra system Sage, and the corresponding code is available in the repository~\cite{FulmanKaplanSinghalWarnaar_SandpileBipartiteRepo}.

Conjecture~\ref{p-Syl_conj} for $p=3$ predicts that for $0<u<1$, $\frac{1}{3}<\alpha \leq 1$, and $\la$ a partition,
\[
\lim_{n\to\infty} \P\big((S_{G_{\alpha}(n,u)})_3 \cong G_{\la}\big) 
= P^{\Sym}_{\infty,3}(\la).
\]
We conducted experiments for $n=100$ and various values of $(\alpha, u)$, with $500$ samples each.
Figure~\ref{fig:3Sylow-bipartite} shows the results for seven partitions $\la$ and compares the observed frequencies with the values predicted by the conjectural distribution $P^{\Sym}_{\infty,3}(\la)$.

\begin{figure}[h]
\begin{center}
\begin{tikzpicture}[scale=0.7,line width=0.35pt]
\draw[->] (0,0)--(15,0);
\draw[->] (0,0)--(0,7.5);
\draw (15.5,0) node {$\lambda$};

\foreach \y in {1,2,...,7} {\draw (0,\y)--(-0.1,\y);}
\foreach \x in {1,3,...,13} {\draw (\x,0)--(\x,-0.1);}

\def\coordinatelist{0.1,0.2,0.3,0.4,0.5,0.6,0.7}
\foreach \p [count=\i from 1] in \coordinatelist 
           {\draw (-0.5,\i) node {$\p$};}

\def\partitions{0,(1),(2),(3),(1,1),(2,1),(3,1)}
\foreach \p [count=\i from 0] in \partitions
           {\draw (2*\i+1,-0.5) node {$\p$};}

\begin{scope}[blue]
\def\bluelist{1.16,0.78,0.34,0.06,0.6,0.12,0.06}
\foreach \p [count=\i from 0] in \bluelist
           {\draw[thick] (2*\i+0.1,\p)--(2*\i+1.8,\p);}
\end{scope}

\begin{scope}[orange]
\def\orangelist{5.8,2.24,0.64,0.26,0.26,0.1,0.1} 
\foreach \p [count=\i from 0] in \orangelist 
           {\draw[thick] (2*\i+0.1,\p)--(2*\i+1.8,\p);}
\end{scope}

\begin{scope}[green]
\def\greenlist{6.2,2.18,0.72,0.28,0.28,0.04,0.06}
\foreach \p [count=\i from 0] in \greenlist 
           {\draw[thick] (2*\i+0.1,\p)--(2*\i+1.8,\p);}
\end{scope}

\begin{scope}[red]
\def\redlist{6.24,2.22,0.68,0.36,0.26,0.04,0}
\foreach \p [count=\i from 0] in \redlist 
           {\draw[thick] (2*\i+0.1,\p)--(2*\i+1.8,\p);}
\end{scope}

\begin{scope}[purple]
\def\purplelist{6.54,1.9,0.64,0.34,0.32,0.12,0.06}
\foreach \p [count=\i from 0] in \purplelist 
           {\draw[thick] (2*\i+0.1,\p)--(2*\i+1.8,\p);}
\end{scope}

\begin{scope}[brown]
\def\brownlist{6.06,2.16,0.76,0.36,0.3,0.08,0.04} 
\foreach \p [count=\i from 0] in \brownlist 
          {\draw[thick] (2*\i+0.1,\p)--(2*\i+1.8,\p);}
\end{scope}

\begin{scope}[pink]
\def\pinklist{6.48,2.2,0.64,0.16,0.3,0.08,0.02}
\foreach \p [count=\i from 0] in \pinklist 
          {\draw[thick] (2*\i+0.1,\p)--(2*\i+1.8,\p);}
\end{scope}

\begin{scope}[gray]
\def\greylist{6.6,1.98,0.82,0.1,0.22,0.08,0.04}
\foreach \p [count=\i from 0] in \greylist 
          {\draw[thick] (2*\i+0.1,\p)--(2*\i+1.8,\p);}
\end{scope}

\begin{scope}[black]
\def\blacklist{6.39,2.13,0.71,0.237,0.266,0.079,0.026} 
\foreach \p [count=\i from 0] in \blacklist 
          {\draw[thick] (2*\i+0.1,\p)--(2*\i+1.8,\p);}
\end{scope}

\draw[thick,black] (13,7)--(13.5,7);
\draw (14,7) node[right] {conjecture};

\def\alphaulist{(0.3,0.5)/blue,(0.4,0.5)/orange,(0.7,0.2)/green,(0.7,0.5)/red,(0.7,0.8)/purple,(1,0.25)/brown,(1,0.5)/pink,(1,0.75)/gray}
\foreach \p/\q [count=\i from 0] in \alphaulist 
          {\draw[thick,\q] (13,6.5-0.5*\i)--(13.5,6.5-0.5*\i);
           \draw (14,6.5-0.5*\i) node[right] {$(\alpha,u)=\p$};}

\end{tikzpicture}
\end{center}
    
\caption{Comparison of the conjectural distribution $P^{\Sym}_{\infty,3}$ and the empirical distribution of $(S_{G_{\alpha}(100,u)})_3$ for various choices of $(\alpha,u)$ and $\lambda$ with sample size $500$.}
    \label{fig:3Sylow-bipartite}
\end{figure}
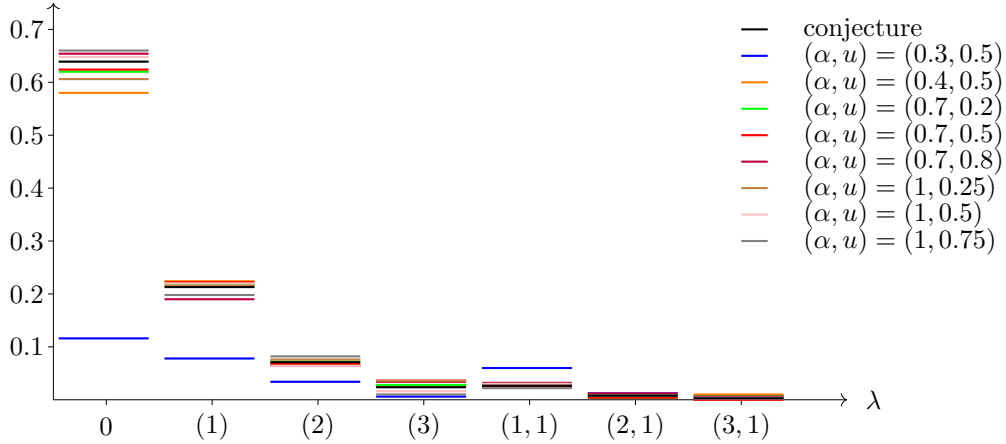

The data corresponding to $\alpha=0.3$ is visibly far from the conjectural distribution
$P^{\Sym}_{\infty,3}$, which is consistent with the fact that $\alpha<\frac{1}{3}=\frac{1}{p}$.
In contrast, for the data with $\alpha>\frac{1}{3}$ the empirical frequencies for these seven partitions are quite close to the conjectural values $P^{\Sym}_{\infty,3}$.

Conjecture~\ref{p-Syl_conj} for $p=2$, predicts that for $0<u<1$, $\frac{1}{2}<\alpha \leq 1$, we have
\[
\lim_{n \to \infty} \P\big((S_{G_{\alpha}(n,u)})_2 \cong G_{\la}\big) 
= P^{\Sym,1}_{2}(\la).
\]
We carried out experiments with $n=100$ and various values of $(\alpha,u)$, taking
$500$ samples for each pair and recording the frequency for which $(S_{G_{\alpha}(n,u)})_2$
is isomorphic to $G_{\la}$ for the same seven partitions $\la$.
Figure~\ref{fig:2Sylow-bipartite} compares these empirical frequencies
with the conjectural values $P^{\Sym,1}_{2}(\la)$.

\begin{figure}[h]
\begin{center}
\begin{tikzpicture}[scale=0.7,line width=0.35pt]
\draw[->] (0,0)--(15,0);
\draw[->] (0,0)--(0,5.5);
\draw (15.5,0) node {$\lambda$};

\foreach \y in {1,2,...,5} {\draw (0,\y)--(-0.1,\y);}
\foreach \x in {1,3,...,13} {\draw (\x,0)--(\x,-0.1);}

\def\coordinatelist{0.05,0.10,0.15,0.20,0.25}
\foreach \p [count=\i from 1] in \coordinatelist 
           {\draw (-0.65,\i) node {$\p$};}

\def\partitions{0,(1),(2),(3),(1,1),(2,1),(3,1)}
\foreach \p [count=\i from 0] in \partitions
           {\draw (2*\i+1,-0.5) node {$\p$};}

\begin{scope}[blue]
\def\bluelist{0.44,1.16,0.36,0.16,0.96,0.4,0.16}
\foreach \p [count=\i from 0] in \bluelist
           {\draw[thick] (2*\i+0.1,\p)--(2*\i+1.8,\p);}
\end{scope}

\begin{scope}[orange]
\def\orangelist{4.32,4.08,2.08,1.48,2.88,1.16,0.36} 
\foreach \p [count=\i from 0] in \orangelist 
           {\draw[thick] (2*\i+0.1,\p)--(2*\i+1.8,\p);}
\end{scope}

\begin{scope}[green]
\def\greenlist{3.92,4.08,1.68,0.8,2.76,1.16,0.92}
\foreach \p [count=\i from 0] in \greenlist 
           {\draw[thick] (2*\i+0.1,\p)--(2*\i+1.8,\p);}
\end{scope}

\begin{scope}[red]
\def\redlist{4.4,4.2,2.04,1,2.64,1.4,0.4}
\foreach \p [count=\i from 0] in \redlist 
           {\draw[thick] (2*\i+0.1,\p)--(2*\i+1.8,\p);}
\end{scope}

\begin{scope}[purple]
\def\purplelist{4.68,4.04,1.72,0.76,2.32,1.12,0.56}
\foreach \p [count=\i from 0] in \purplelist 
           {\draw[thick] (2*\i+0.1,\p)--(2*\i+1.8,\p);}
\end{scope}

\begin{scope}[brown]
\def\brownlist{4.28,4.52,1.84,0.92,2.6,1.08,0.48}
\foreach \p [count=\i from 0] in \brownlist 
          {\draw[thick] (2*\i+0.1,\p)--(2*\i+1.8,\p);}
\end{scope}

\begin{scope}[pink]
\def\pinklist{4.48,3.8,1.96,1.2,2.72,1.52,0.4}
\foreach \p [count=\i from 0] in \pinklist 
          {\draw[thick] (2*\i+0.1,\p)--(2*\i+1.8,\p);}
\end{scope}

\begin{scope}[gray]
\def\greylist{4.6,3.68,2.6,0.88,2.84,1.24,0.52}
\foreach \p [count=\i from 0] in \greylist 
          {\draw[thick] (2*\i+0.1,\p)--(2*\i+1.8,\p);}
\end{scope}

\begin{scope}[black]
\def\blacklist{4.194,4.194,2.098,1.048,2.796,1.048,0.524} 
\foreach \p [count=\i from 0] in \blacklist 
          {\draw[thick] (2*\i+0.1,\p)--(2*\i+1.8,\p);}
\end{scope}

\draw[thick,black] (13,6.5)--(13.5,6.5);
\draw (14,6.5) node[right] {conjecture};

\def\alphaulist{(0.45,0.5)/blue,(0.6,0.5)/orange,(0.7,0.2)/green,(0.7,0.8)/red,(0.8,0.5)/purple,(1,0.25)/brown,(1,0.5)/pink,(1,0.75)/gray}
\foreach \p/\q [count=\i from 0] in \alphaulist 
          {\draw[thick,\q] (13,6-0.5*\i)--(13.5,6-0.5*\i);
           \draw (14,6-0.5*\i) node[right] {$(\alpha,u)=\p$};}

\end{tikzpicture}
\end{center}
    
\caption{Comparison of the conjectural distribution $P^{\Sym,1}_2$ and the empirical distribution of $(S_{G_{\alpha}(100,u)})_2$ for various choices of $(\alpha,u)$ and $\lambda$ with sample size $500$.}
    \label{fig:2Sylow-bipartite}
\end{figure}
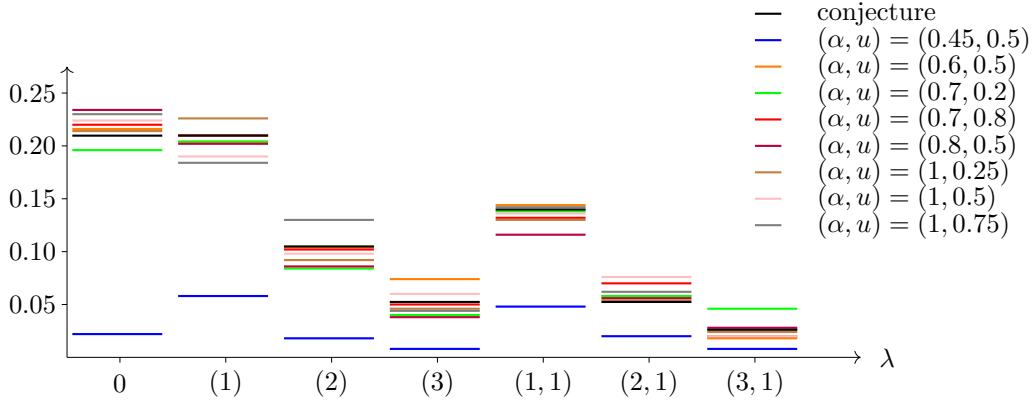

The data corresponding to $\alpha=0.45$ is far from the conjectural distribution
$P^{\Sym,1}_2$, which is consistent with the fact that $\alpha<1/2=1/p$.
In contrast, for the data with $\alpha>1/2$ the empirical frequencies for these seven partitions are all close to the conjectural values predicted by the distribution $P^{\Sym,1}_2$.

We determine the probability that the distribution $\spmk$ chooses a partition of a specified rank. We begin with a result of Wood.

\begin{prop}[{\!\!\cite[Corollary 9.4]{wood2017distribution}}]\label{wood_rank_dist}
If $\la$ is a random partition chosen according to $\msympinf$, then
\[
\P_{\msympinf}\big(\la_1'=r\big)
=\frac{1}{(-q;q)_{\infty}}\,\frac{q^{\binom{r+1}{2}}}{(q;q)_r}.
\]
\end{prop}

Proposition~\ref{wood_rank_dist} implies the following result.

\begin{prop}\label{rank_dist}
The probability $\mathfrak{p}_{p,k}(r)$ that a partition chosen from the distribution $\spmk$ has $r$ parts is
\[
\mathfrak{p}_{p,k}(r)=\frac{1}{(-1;q)_{k}(-q;q)_{\infty}}\,\frac{q^{\binom{r-k+1}{2}}}{(q;q)_r}.
\]
\end{prop}

Conjecture~\ref{p-Syl_conj} implies the following conjecture.

\begin{conj}\label{p-Syl_ranks_conj}
Let $p$ be prime, $\frac{1}{p}<\alpha\leq 1$, and $r$ a nonnegative integer.
\begin{enumerate}
\item If $p$ is odd, then
\[
\lim_{n\to\infty} \P\big(\rank((S_{G_{\alpha}(n,u)})_p) = r\big)
=\mathfrak{p}_{p,0}(r) 
=\frac{1}{(-q;q)_{\infty}}\,\frac{q^{\binom{r+1}{2}}}{(q;q)_r}.
\]
\item If $p=2$, then
\[
\lim_{n \to \infty} \P\big(\rank((S_{G_{\alpha}(n,u)})_2) = r\big) 
= \mathfrak{p}_{2,1}(r)
=\frac{1}{(-1;q)_{\infty}} \,\frac{q^{\binom{r}{2}}}{(q;q)_r}.
\]
\end{enumerate}
\end{conj}

For various values of $\alpha$ and $u$, we randomly generated $500$ bipartite graphs $G_{\alpha}(n,u)$ with $n=100$ to evaluate the $2$-ranks of their sandpile groups.  The resulting distributions, along with the conjectured distribution from Conjecture~\ref{p-Syl_conj} are plotted in Figure~\ref{fig: 2-rank}.
For $\alpha>1/2$ and all values of $u$, there is close agreement with the conjectured distribution. We then conducted a parallel experiment, generating a fresh sample of $500$ graphs for each $(\alpha,u)$ pair to independently evaluate the $3$-ranks.  Figure~\ref{fig: 3-rank} plots these results.  As in the case $p=2$, we see a clear threshold phenomenon in the parameter $\alpha$.  For $\alpha>1/3$, the empirical $3$-rank distributions track the conjectural curve quite closely.  In contrast, when $\alpha$ falls below the $1/p$ threshold (i.e., $\alpha = 0.45$ for $p=2$, and $\alpha = 0.3$ for $p=3$), the empirical distributions place much less mass near small ranks and exhibit noticeably heavier tails.

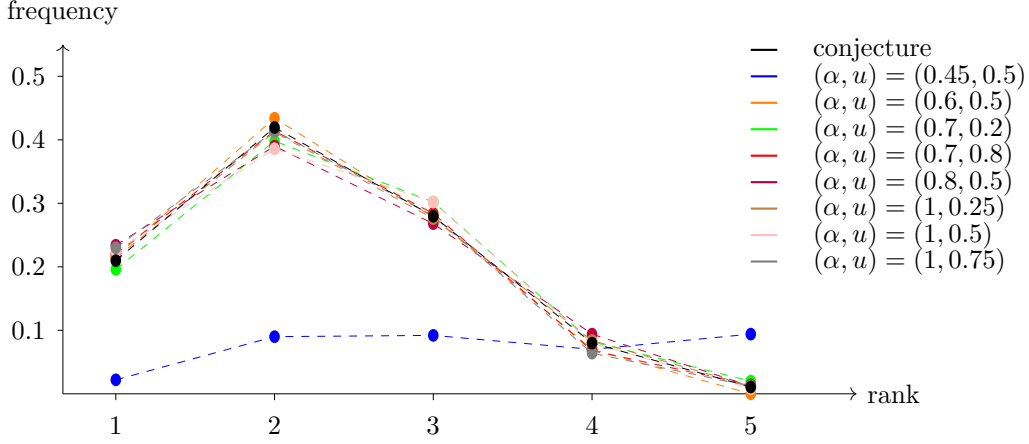
\begin{figure}[h]
\begin{center}
\begin{tikzpicture}[scale=0.7,line width=0.35pt]
\begin{scope}[yscale=1.2]
\draw[->] (0,0)--(15,0);
\draw[->] (0,0)--(0,5.5);
\draw (15.7,0) node {$\rank$};
\draw (0,6) node {frequency};

\foreach \y in {1,2,...,5} {\draw (0,\y)--(-0.1,\y);}
\foreach \x in {1,4,...,13} {\draw (\x,0)--(\x,-0.1);}

\def\coordinatelist{0.1,0.2,0.3,0.4,0.5}
\foreach \p [count=\i from 1] in \coordinatelist 
           {\draw (-0.65,\i) node {$\p$};}

\def\ranklist{1,2,3,4,5}
\foreach \p [count=\i from 0] in \ranklist
           {\draw (3*\i+1,-0.5) node {$\p$};}

\begin{scope}[blue]
\def\blacklist{0.22,0.9,0.92,0.7,0.94} 
\foreach \p [count=\i from 0] in \blacklist 
          {\fill (3*\i+1,\p) circle (1mm);}
\draw[dashed] (1,0.22)--(4,0.9)--(7,0.92)--(10,0.7)--(13,0.94);
\end{scope}

\begin{scope}[orange]
\def\orangelist{2.16,4.34,2.78,0.7,0} 
\foreach \p [count=\i from 0] in \orangelist
          {\fill (3*\i+1,\p) circle (1mm);}
\draw[dashed] (1,2.16)--(4,4.34)--(7,2.78)--(10,0.7)--(13,0);
\end{scope}

\begin{scope}[green]
\def\greenlist{1.96,3.98,3.02,0.82,0.2}
\foreach \p [count=\i from 0] in \greenlist 
          {\fill (3*\i+1,\p) circle (1mm);}
\draw[dashed] (1,1.96)--(4,3.98)--(7,3.02)--(10,0.82)--(13,0.2);
\end{scope}

\begin{scope}[red]
\def\redlist{2.2,4.14,2.84,0.68,0.14}
\foreach \p [count=\i from 0] in \redlist 
          {\fill (3*\i+1,\p) circle (1mm);}
\draw[dashed] (1,2.2)--(4,4.14)--(7,2.84)--(10,0.68)--(13,0.14);
\end{scope}

\begin{scope}[purple]
\def\purplelist{2.34,3.9,2.68,0.94,0.12}
\foreach \p [count=\i from 0] in \purplelist 
          {\fill (3*\i+1,\p) circle (1mm);}
\draw[dashed] (1,2.34)--(4,3.9)--(7,2.68)--(10,0.94)--(13,0.12);
\end{scope}

\begin{scope}[brown]
\def\brownlist{2.14,4.12,2.76,0.84,0.12}
\foreach \p [count=\i from 0] in \brownlist 
          {\fill (3*\i+1,\p) circle (1mm);}
\draw[dashed] (1,2.14)--(4,4.12)--(7,2.76)--(10,0.84)--(13,0.12);
\end{scope}

\begin{scope}[pink]
\def\pinklist{2.24,3.86,3.02,0.8,0.08}
\foreach \p [count=\i from 0] in \pinklist 
          {\fill (3*\i+1,\p) circle (1mm);}
\draw[dashed] (1,2.24)--(4,3.86)--(7,3.02)--(10,0.8)--(13,0.08);
\end{scope}

\begin{scope}[gray]
\def\greylist{2.3,4.14,2.8,0.64,0.12}
\foreach \p [count=\i from 0] in \greylist 
          {\fill (3*\i+1,\p) circle (1mm);}
\draw[dashed] (1,2.3)--(4,4.14)--(7,2.8)--(10,0.64)--(13,0.12);
\end{scope}

\begin{scope}[black]
\def\blacklist{2.09711,4.19422,2.79615,0.79890,0.10652} 
\foreach \p [count=\i from 0] in \blacklist 
          {\fill (3*\i+1,\p) circle (1mm);}
\draw[dashed] (1,2.09711)--(4,4.19422)--(7,2.79615)--(10,0.7989)--(13,0.10652);
\end{scope}
\end{scope}

\draw[thick,black] (13,6.5)--(13.5,6.5);
\draw (14,6.5) node[right] {conjecture};

\def\alphaulist{(0.45,0.5)/blue,(0.6,0.5)/orange,(0.7,0.2)/green,(0.7,0.8)/red,(0.8,0.5)/purple,(1,0.25)/brown,(1,0.5)/pink,(1,0.75)/gray}
\foreach \p/\q [count=\i from 0] in \alphaulist 
          {\draw[thick,\q] (13,6-0.5*\i)--(13.5,6-0.5*\i);
           \draw (14,6-0.5*\i) node[right] {$(\alpha,u)=\p$};}

\end{tikzpicture}
\end{center}

\caption{Distribution of the $2$-rank of the sandpile group of a random bipartite graph.}
    \label{fig: 2-rank}
\end{figure}

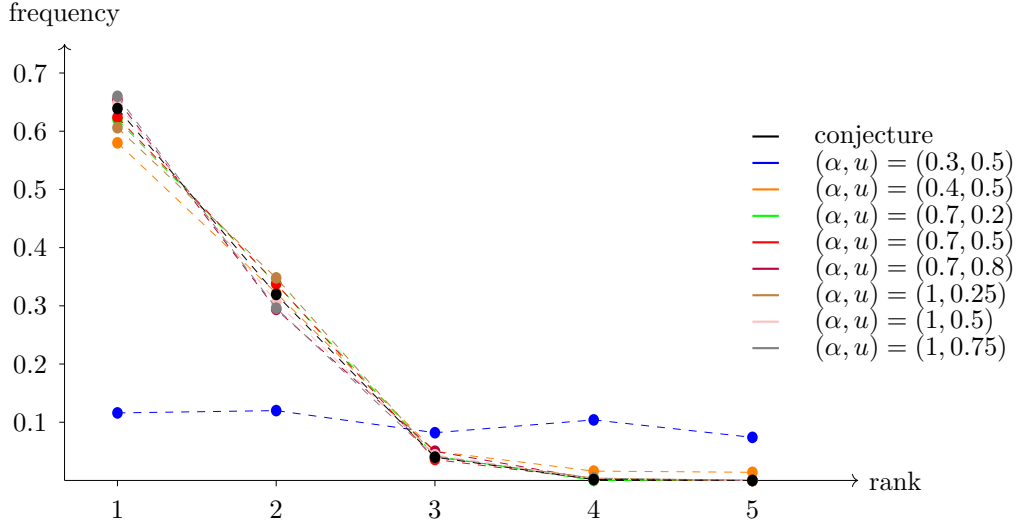
\begin{figure}[h]
\begin{center}
\begin{tikzpicture}[scale=0.7,line width=0.35pt]
\begin{scope}[yscale=1.1]
\draw[->] (0,0)--(15,0);
\draw[->] (0,0)--(0,7.5);
\draw (15.7,0) node {$\rank$};
\draw (0,8) node {frequency};

\foreach \y in {1,2,...,7} {\draw (0,\y)--(-0.1,\y);}
\foreach \x in {1,4,...,13} {\draw (\x,0)--(\x,-0.1);}

\def\coordinatelist{0.1,0.2,0.3,0.4,0.5,0.6,0.7}
\foreach \p [count=\i from 1] in \coordinatelist 
           {\draw (-0.65,\i) node {$\p$};}

\def\ranklist{1,2,3,4,5}
\foreach \p [count=\i from 0] in \ranklist
           {\draw (3*\i+1,-0.5) node {$\p$};}

\begin{scope}[blue]
\def\blacklist{1.16,1.2,0.82,1.04,0.74} 
\foreach \p [count=\i from 0] in \blacklist 
          {\fill (3*\i+1,\p) circle (1mm);}
\draw[dashed] (1,1.16)--(4,1.2)--(7,0.82)--(10,1.04)--(13,0.74);
\end{scope}

\begin{scope}[orange]
\def\orangelist{5.8,3.22,0.5,0.16,0.14} 
\foreach \p [count=\i from 0] in \orangelist
          {\fill (3*\i+1,\p) circle (1mm);}
\draw[dashed] (1,5.8)--(4,3.22)--(7,0.5)--(10,0.16)--(13,0.14);
\end{scope}

\begin{scope}[green]
\def\greenlist{6.2,3.38,0.42,0,0}
\foreach \p [count=\i from 0] in \greenlist 
          {\fill (3*\i+1,\p) circle (1mm);}
\draw[dashed] (1,6.2)--(4,3.38)--(7,0.42)--(10,0)--(13,0);
\end{scope}

\begin{scope}[red]
\def\redlist{6.24,3.38,0.36,0.02,0}
\foreach \p [count=\i from 0] in \redlist 
          {\fill (3*\i+1,\p) circle (1mm);}
\draw[dashed] (1,6.24)--(4,3.38)--(7,0.36)--(10,0.02)--(13,0);
\end{scope}

\begin{scope}[purple]
\def\purplelist{6.54,2.94,0.5,0.02,0}
\foreach \p [count=\i from 0] in \purplelist 
          {\fill (3*\i+1,\p) circle (1mm);}
\draw[dashed] (1,6.54)--(4,2.94)--(7,0.5)--(10,0.02)--(13,0);
\end{scope}

\begin{scope}[brown]
\def\brownlist{6.06,3.48,0.42,0.04,0}
\foreach \p [count=\i from 0] in \brownlist 
          {\fill (3*\i+1,\p) circle (1mm);}
\draw[dashed] (1,6.06)--(4,3.48)--(7,0.42)--(10,0.04)--(13,0);
\end{scope}

\begin{scope}[pink]
\def\pinklist{6.48,3.06,0.44,0.02,0}
\foreach \p [count=\i from 0] in \pinklist 
          {\fill (3*\i+1,\p) circle (1mm);}
\draw[dashed] (1,6.48)--(4,3.06)--(7,0.44)--(10,0.02)--(13,0);
\end{scope}

\begin{scope}[gray]
\def\greylist{6.6,2.96,0.4,0.04,0}
\foreach \p [count=\i from 0] in \greylist 
          {\fill (3*\i+1,\p) circle (1mm);}
\draw[dashed] (1,6.6)--(4,2.96)--(7,0.4)--(10,0.04)--(13,0);
\end{scope}

\begin{scope}[black]
\def\blacklist{6.39005,3.19502,0.39938,0.01536,0.00019} 
\foreach \p [count=\i from 0] in \blacklist 
          {\fill (3*\i+1,\p) circle (1mm);}
\draw[dashed] (1,6.39005)--(4,3.19502)--(7,0.39938)--(10,0.01536)--(13,0.00019);
\end{scope}
\end{scope}

\draw[thick,black] (13,6.5)--(13.5,6.5);
\draw (14,6.5) node[right] {conjecture};

\def\alphaulist{
(0.3,0.5)/blue,(0.4,0.5)/orange,(0.7,0.2)/green,(0.7,0.5)/red,(0.7,0.8)/purple,(1,0.25)/brown,(1,0.5)/pink,(1,0.75)/gray}
\foreach \p/\q [count=\i from 0] in \alphaulist 
          {\draw[thick,\q] (13,6-0.5*\i)--(13.5,6-0.5*\i);
           \draw (14,6-0.5*\i) node[right] {$(\alpha,u)=\p$};}

\end{tikzpicture}
\end{center}

\caption{Distribution of the $3$-rank of the sandpile group of random bipartite graphs.}
    \label{fig: 3-rank}
\end{figure}

Next, we compute the expected value of $p^{l(\la)}$, for a random partition $\la$ coming from $\spmk$.
This is a natural quantity to consider as $p^{\rank((S_{G_{\alpha}(n,u)})_p)}=\abs{S_{G_{\alpha}(n,u)}/ pS_{G_{\alpha}(n,u)}}$.

\begin{lemma}\label{Lem: Exp p^p rank}
Let $\la$ be a random partition chosen from the distribution $\spmk$.
The expected value of $p^{l(\la)}$ is $1+p^k$.    
\end{lemma}

\begin{proof}[Proof of Lemma~\ref{Lem: Exp p^p rank} assuming Theorem~\ref{Thm_r=1_one-par-family}]
Note that
\[
p^{l(\la)} =\abs{\Hom(G_{\la},G_{(1)})} =\abs{\Sur(G_{\la},G_{(1)})} + \abs{\Sur(G_{\la},G_0)}.\]
Theorem~\ref{Thm_r=1_one-par-family} implies that $\E(\abs{\Sur(G_{\la},G_{(1)})}) =p^k$ and $\E(\abs{\Sur(G_{\la},G_0)}) =1$.
\end{proof}

Conjecture~\ref{p-Syl_ranks_conj} thus suggests that the following should hold:
\begin{enumerate}
    \item If $p$ is an odd prime and $\frac{1}{p}<\alpha\leq 1$, then 
    \[
    \lim_{n\to\infty} \E\big(p^{\rank((S_{G_{\alpha}(n,u)})_p)}\big) = 2.
    \]
    \item If $p=2$ and $\frac{1}{2}<\alpha\leq 1$, then
    \[
    \lim_{n \to \infty} \E\big(2^{\rank((S_{G_{\alpha}(n,u)})_2)}\big) = 3.
    \]
\end{enumerate}
However, based on our experiments this might not always be correct. 
Perhaps these limits hold for $\beta<\alpha\leq 1$ for some $\beta>1/p$.

Table~\ref{tab:2Sylow-expectation} summarizes the results for $p=2$ that were observed for $500$ random graphs for each $(\alpha, u)$.
Notice that $\alpha=0.45<1/2$ produces an extremely large empirical value.  Because the expected values for small $\alpha$ diverge so massively from the asymptotic predictions, we display the results in Table~\ref{tab:2Sylow-expectation} rather than in a graph to explicitly highlight the scale of these extreme values. 
For $\alpha>1/2$, the empirical expectations are already reasonably close to $3$.

\begin{table}[ht]
\centering
\resizebox{\textwidth}{!}{
\begin{tabular}{|c|c|c|c|c|c|c|c|c|}
\hline
 $(\alpha, u)$ & $(0.45,0.5)$ & $(0.6,0.5)$ & $(0.8,0.5)$ & $(0.7,0.2)$ & $(0.7,0.8)$ & $(1,0.25)$ & $(1,0.5)$ & $(1,0.75)$ \\
\hline
$\mathbb{E}(2^{\mathrm{rank}})$ & 11382.922 & 2.820 & 3.094 & 3.240 & 2.952 & 3.070 & 2.972 & 2.882 \\
\hline
\end{tabular}
}
\caption{Observed values of $\mathbb{E}(2^{\mathrm{rank}})$ for $(S_{G_{\alpha}(n,u)})_2$ with $n=100$ and $500$ samples for each $(\alpha,u)$.}
\label{tab:2Sylow-expectation}
\end{table}

We ran a parallel experiment for $p=3$. Table~\ref{tab:3Sylow-expectation} reports the empirical values obtained from our simulations of $500$ samples for several $(\alpha,u)$. 

In the experiment with $\alpha=0.4$, and $u=0.5$ we obtained
$\mathbb{E}(3^{\rank})=1110.184\dots$ from $500$ sampled graphs.
This is almost entirely explained by two unusual graphs: one whose
Sylow $3$-subgroup has $3$-rank $12$ ($\la=(1^{12})$) and
one with $3$-rank $9$ ($\la=(2,1^{8})$).
These two graphs contribute $1102.25\dots$ in the empirical average.
The remaining $498$ graphs contribute only $7.97\dots$ to the average, which is still larger than $2$ but orders of magnitude smaller than the full value $1110.184$. 
If we further exclude four graphs with $3$-rank $5$, then the observed average comes down to $2.14\dots$.

For $\alpha=0.3$ the observed expectation is much larger than $2$, which may be understood since $0.3<1/3$.
This time the large deviation from $2$ is not caused by a few outliers; $289$ out of the $500$ graphs in our sample had a $3$-rank of at least $4$.  For larger $\alpha$ the observed expectations are all close to $2$.

\begin{table}[ht]
\centering
\resizebox{\textwidth}{!}{
\begin{tabular}{|c|c|c|c|c|c|c|c|c|}
\hline
 $(\alpha,u)$ & $(0.3,0.5)$ & $(0.4,0.5)$ & $(0.7,0.5)$ & $(0.7,0.2)$ & $(0.7,0.8)$ & $(1,0.25)$ & $(1,0.5)$ & $(1,0.75)$ \\
\hline
$\mathbb{E}(3^{\mathrm{rank}})$ &  71873408 & 1110 & 2.016 & 2.012 & 2.040 & 2.136 & 2.016 & 2.016 \\
\hline
\end{tabular}
}
\caption{Observed values of $\mathbb{E}(3^{\mathrm{rank}})$ for $(S_{G_{\alpha}(n,u)})_3$ with $n=100$ and $500$ samples for each $(\alpha,u)$.}
\label{tab:3Sylow-expectation}
\end{table}

\section{Preliminary material}\label{sec:Background}

\subsection{Hall--Littlewood symmetric polynomials}\label{Subsec: defn of sym polynomial}

A crucial ingredient in our computations of moments is the theory of Hall--Littlewood symmetric functions. These have numerous applications, such as to the representation theory of the general linear group over a finite field.
A comprehensive treatment is given by Macdonald in \cite[Chapter 3]{macdonald1995symmetric} and we adhere to Macdonald's notation in this paper, except for the use of $q$ instead of $t$ in the definition of Hall--Littlewood polynomials.

Let $\mathbb{Q}(q)[x_1,\dots,x_n]^{S_n}$ be the ring of symmetric polynomials in $n$ variables with coefficients in $\mathbb{Q}(q)$.
Two important bases in this ring are given by the Hall--Littlewood polynomials $P_{\la}(q)=P_{\la}(x_1,\dots,x_n;q)$ and $Q_{\la}(q)=Q_{\la}(x_1,\dots,x_n;q)$, where $\la$ ranges over all partitions of length at most $n$.
To define these polynomials, let $\la$ be a partition such that $l(\la)\leq n$, and let
\[
v_{\la}(q):= \frac{(q;q)_{n-l(\la)}}{(1-q)^{n-l(\la)}} 
\prod_{i\geq 1} \frac{(q;q)_{m_i(\la)}}{(1-q)^{m_i(\la)}}.
\]
Then
\[
P_{\la}(x_1,\dots,x_n;q) 
:=\frac{1}{v_{\la}(q)} \sum_{w\in S_n} w\bigg( x_1^{\la_1}\dots x_n^{\la_n} \prod_{i<j} \frac{x_i-q x_j}{x_i-x_j}\bigg),
\]
where $S_n$ acts by permuting the variables $x_1,\dots,x_n$.
To define the second family of Hall--Littlewood polynomials we require the
multiplicative factor
\[
b_{\la}(q):=\prod_{i\geq 1} (q;q)_{m_i(\la)},
\]
which then gives
\begin{equation}\label{QP}
Q_{\la}(q):=b_{\la}(q) P_{\la}(q).
\end{equation}
Thanks to the stability property
\begin{equation}\label{stable}
P_{\la}(x_1,\dots,x_{n-1},0;q)=
\begin{cases}
P_{\la}(x_1,\dots,x_{n-1};q) & \text{if $l(\la)\leq n-1$}, \\
0 & \text{otherwise},
\end{cases}
\end{equation}
the Hall--Littlewood polynomials may be extended to symmetric functions in countably many variables $x_1,x_2,\dots$.
In particular,
\[
P_{\la}(x_1,\dots,x_n,0,0,\dots;q)=
\begin{cases}
P_{\la}(x_1,\dots,x_n;q) & \text{if $l(\la)\leq n$}, \\
0 & \text{otherwise},
\end{cases}
\]
in accordance with \eqref{stable}.

Given the two families of Hall--Littlewood polynomials,
we define the structure constants $f_{\mu\nu}^{\la}(q)$ and $g_{\mu\nu}^{\la}(q)$ by
\[
P_{\mu}(q)\,P_{\nu}(q)=\sum_{\la} f_{\mu\nu}^{\la}(q)\,P_\la(q)
\quad\text{and}\quad
Q_{\mu}(q)\,Q_{\nu}(q)=\sum_{\la} g_{\mu\nu}^{\la}(q)\,Q_\la(q).
\]
The \emph{skew} Hall–-Littlewood polynomials are then defined by
\[
P_{\la/\mu}(q):=\sum_{\nu} f_{\mu\nu}^{\la}(q)\,P_{\nu}(q)
\quad\text{and}\quad
Q_{\la/\mu}(q):=\sum_{\nu} g_{\mu\nu}^{\la}(q)\,Q_\nu(q),
\]
which are zero unless $\mu\subseteq\la$.
Equivalently, using $Q_{\la}(q)=b_{\la}(q)\,P_{\la}(q)$ with $b_{\la}(q)$ as above, one has
\[
g_{\mu\nu}^{\la}(q)=\frac{b_{\la}(q)}{b_{\mu}(q)\,b_{\nu}(q)}\,f_{\mu\nu}^{\la}(q),
\]
and hence
\begin{equation}\label{Eqn: Skew Q to P}
Q_{\la/\mu}(q)
=\frac{b_{\la}(q)}{b_{\mu}(q)}\,P_{\la/\mu}(q).
\end{equation}
This function admits the principal specialization formula (see \cite[Equation~(4.3)]{warnaar2012dedekind})
\begin{equation}\label{Eqn: Q la minus mu}
Q_{\la/\mu}\big(1,q,q^2,\dots;q\big)
=q^{n(\la/\mu)}\prod_{i\geq 1} 
\qbin{\la'_i-\mu'_{i+1}}{\la'_i-\mu'_i}_q,
\end{equation}
which for $\mu=0$ simplifies to \cite[page~213]{macdonald1995symmetric}
\begin{equation}\label{Eqn: Q la 1 q q2}
Q_{\la}\big(1,q,q^2,\dots;q\big)=q^{n(\la)}.
\end{equation}

A final result needed in this paper is the skew Cauchy identity \cite[page~227]{macdonald1995symmetric}
\begin{equation}\label{Eqn: Skew Cauchy identity}
\sum_{\la} Q_{\la}(x;q) P_{\la/\mu}(y;q)=
Q_{\mu}(x;q) \prod_{i,j\geq 1} \frac{1-qx_iy_j}{1-x_iy_j},
\end{equation}
where $x=(x_1,x_2,\dots)$ and $y=(y_1,y_2,\dots)$ are alphabets of countably-many variables.

\subsection{$q$-Hypergeometric series}

Besides the use of Hall--Littlewood symmetric functions, many of our proofs require results from the theory of \emph{basic hypergeometric series}.
In this section, we recall the definition of such series and provide a list of summation and transformation formulas used in the remainder of the paper.

Let $(a_1,\dots,a_k;q)_n:=(a_1;q)_n\cdots (a_k;q)_n$.
Then, for nonnegative integers $r,s$, the
${_r\phi_s}$ basic hypergeometric series is defined as \cite{GR04}
\begin{align*}
\qhyp{r}{s}{a_1,\dots,a_r}{b_1,\dots,b_s}{z,q}&\hphantom{:}=
{_r\phi_s}(a_1,\dots,a_r;b_1,\dots,b_s;z,q) \\
&:=\sum_{k=0}^{\infty} \frac{(a_1,\dots,a_r;q)_k}{(q,b_1,\dots,b_s;q)_k}
\Big((-1)^k q^{\binom{k}{2}}\Big)^{s-r+1} z^k,
\end{align*}
where it is assumed that the $b_i$ are in general position 
in $\mathbb{C}$.
We exclusively use terminating series for which convergence is not an issue. 

First we list three well-known summation formulas.
The terminating form of the $q$-binomial theorem 
\cite[Equation (II.4)]{GR04} is
\begin{equation}\label{qbt}
\qhyp{1}{0}{q^{-n}}{\text{--}}{q,zq^n}=
\sum_{k=0}^n (-z)^k q^{\binom{k}{2}} \qbin{n}{k}_q=(z;q)_n,
\end{equation}
where, here and in the following, $n$ is a nonnegative integer.
Our second summation is the terminating form of the 
${_1\phi_1}$ summation \cite[Equation (II.5)]{GR04}
\begin{equation}\label{Eqn: 1phi1}
\qhyp{1}{1}{q^{-n}}{c}{q,cq^n}=\frac{1}{(c;q)_n},
\end{equation}
and out third and final summation corresponds to the known evaluation
of the Rogers--Szeg\H{o} polynomial $H_n(z;q)$ at $z=-q$ \cite[Theorem 8.1]{berkovich2005positivity}
\begin{equation}\label{RS}
\sum_{k=0}^n (-q)^k \qbin{n}{k}_q=(q;q^2)_{\ceil{n/2}}.
\end{equation}
We further require three transformation formulas for basic hypergeometric series, all of which are limits of more general such formulas.
Taking the $b\to 0$ limits in 
\cite[Equation~(III.6)]{GR04} and \cite[Equation~(III.8)]{GR04} yields
\begin{subequations}
\begin{align}\label{III6}
\qhyp{2}{1}{0,q^{-n}}{c}{q,z}&=\frac{q^{\binom{n}{2}}}{(c;q)_n}
\Big({-}\frac{cz}{q}\Big)^n \,\qhyp{3}{2}{q^{-n},q/z,q^{1-n}/c}{0,0}{q,q} \\
&=\frac{(-c)^n q^{\binom{n}{2}}}{(c;q)_n}\,
\qhyp{2}{0}{q^{-n},q/z}{\text{--}}{q,\frac{z}{c}}.
\label{III8}
\end{align}
\end{subequations}
Equating the two expressions on the right and replacing $z\mapsto q/b$ followed by $c\mapsto q/bz$ results in
\begin{equation}\label{III7}
\qhyp{2}{0}{b,q^{-n}}{\text{--}}{q,z}=b^{-n}
\qhyp{3}{2}{b,bzq^{-n},q^{-n}}{0,0}{q,q}.
\end{equation}
This may also be obtained from \cite[Equation~(III.7)]{GR04} by replacing $z$ by $cz$ and then letting $c$ tend to infinity.

Finally, we need a multiple basic hypergeometric series of Srivastava--Daoust or Kamp\'e de F\'eri\'et type. For $\kappa$ a partition of length at most $r$ we define \cite{bytev2021q}
\begin{equation}\label{SD}
\mathcal{F}^{\kappa}(x_2,\dots,x_r)
:=\sum_{l_2=0}^{\kappa_2}\cdots\sum_{l_r=0}^{\kappa_r}
\prod_{i=1}^r
\frac{(q^{-\kappa_i};q)_{l_i+l_{i+1}}} {(q,-1,-q^{1-\kappa_{i-1}},q^{-\kappa_{i-1}};q)_{l_i}}\,x_i^{l_i},
\end{equation}
where $l_1=l_{r+1}:=0$.
Each $l_i$ (for $2\leq i\leq r$) is summed from $0$ to $\kappa_i$.
Since $\kappa_i\leq\kappa_{i-1}$, the denominator factor $(q^{-\kappa_{i-1}};q)_{l_i}$ does not cause any issues and $\mathcal{F}^{\kappa}$ is well defined.
As some special cases, we note that for $r=1$ we have
$\mathcal{F}^{(\kappa_1)}=1$ and for $r=2$,
\[
\mathcal{F}^{(\kappa_1,\kappa_2)}(x_2)={_3\phi_2}(0,0,q^{-\kappa_2};-1,-q^{1-\kappa_1};q,x_2).
\]
Furthermore, we have the stability property
\[
\mathcal{F}^{(\kappa_1,\dots,\kappa_{r-1},0)}(x_2,\dots,x_r)
=\mathcal{F}^{(\kappa_1,\dots,\kappa_{r-1})}(x_2,\dots,x_{r-1}),
\]
so that $\mathcal{F}^{\kappa}(x_2,\dots,x_r)$ depends only on the variables
$x_2,\dots,x_{l(\kappa)}$.

\section{Computing moments}\label{Sec: computing moments}

\subsection{Families of measures with the same moments}

We first rewrite $\msympinf$ in terms of Markov chain transition probabilities. 
For integers $a,b$ let
\begin{equation}\label{Kab}
K(a,b):=q^{\binom{b+1}{2}} (q;q^2)_{\ceil{\frac{a-b}{2}}} \qbin{a}{b}_q,
\end{equation}
so that $K(a,b)=0$ unless $0\leq b\leq a$.
Further set
\begin{equation}\label{Kinfb}
K(\infty,b):=\lim_{a\to\infty} K(a,b)
=(q;q^2)_{\infty}\,\frac{q^{\binom{b+1}{2}}}{(q;q)_b}.
\end{equation}
It was shown in \cite[Theorem 4.1]{fulman2019random} that
\begin{equation}\label{F2019}
\msympinf(\la)=K(\infty,\la_1') 
\prod_{i\geq 1} K(\la_i',\la_{i+1}').
\end{equation}

\begin{lemma}\label{Kabsum}
For a nonnegative integer $a$,
\[
\sum_{b=0}^a K(a,b)=1.
\]
\end{lemma}

We defer the proof to Section~\ref{Subsec: Technical Lemmas K(a,b)}.
The lemma implies that for a fixed nonnegative integer $a$, $K(a,b)$ gives a distribution on $\{0,1,\dots,a\}$. 
We can interpret \cite[Theorem 4.1]{fulman2019random} as saying that with respect to the distribution $\msympinf$, we have the conditional probability $\P(\la'_{i+1}=b\mid \la'_i=a)=K(a,b)$.
Since $q<1$, by the dominated convergence theorem, we immediately obtain
the following corollary. 

\begin{lemma}\label{Lem: sum K(infty,b) is 1}
We have
\[
\sum_{b=0}^{\infty} K(\infty,b)=1.
\]
\end{lemma}

Moreover, we can express $\msympinf(\la) \abs{\Sur_p(\la,\mu)}$ as a product over the function $K(a,b)$.

\begin{lemma}\label{Lem: PSym inf times sur product}
Given partitions $\mu\subseteq\la$,
\[
\msympinf(\la) \abs{\Sur_p(\la,\mu)}
= q^{-n(\mu)}
K(\infty,\la'_1-\mu'_1)
\prod_{i\geq 1} K(\la'_i-\mu'_{i+1},\la'_{i+1}-\mu'_{i+1}).
\]
\end{lemma}

\begin{proof}
We begin by noting that by \eqref{QP}, \eqref{Eqn: Q la 1 q q2} and the homogeneity of the Hall--Littlewood symmetric functions, Proposition~\ref{surjhall} may be restated as
\begin{equation}\label{Sur1}
\abs{\Sur_p(\la,\mu)}
=q^{-n(\mu)-n(\la)-\abs{\mu}} b_{\la}(q) P_{\la/\mu}(1,q,q^2,\dots;q).
\end{equation}
Moreover, by \eqref{Eqn: Skew Q to P} and \eqref{Eqn: Q la minus mu},
\[
P_{\la/\mu}(1,q,q^2,\dots;q)=q^{n(\la/\mu)} \frac{b_{\mu}(q)}{b_{\la}(q)}
\prod_{i\geq 1} \qbin{\la'_i-\mu'_{i+1}}{\la'_i-\mu'_i}_q,
\]
and thus
\[
\abs{\Sur_p(\la,\mu)}
=q^{n(\la/\mu)-n(\la)-n(\mu)-\abs{\mu}} b_{\mu}(q)
\prod_{i\geq 1} \qbin{\la'_i-\mu'_{i+1}}{\la'_i-\mu'_i}_q.
\]
By the readily verified
\[
\prod_{i\geq 1} \qbin{\la'_i-\mu'_{i+1}}{\la'_i-\mu'_i}_q
=\frac{b_{\la}(q)}{b_{\mu}(q)}\,\frac{1}{(q;q)_{\la'_1-\mu'_1}}
\prod_{i\geq 1} \qbin{\la'_i-\mu'_{i+1}}{\la'_{i+1}-\mu'_{i+1}}_q,
\]
this leads to
\begin{equation}\label{Sur2}
\abs{\Sur_p(\la,\mu)}
=q^{n(\la/\mu)-n(\la)-n(\mu)-\abs{\mu}} \frac{b_{\la}(q)}{(q;q)_{\la'_1-\mu'_1}}
\prod_{i\geq 1} \qbin{\la'_i-\mu'_{i+1}}{\la'_{i+1}-\mu'_{i+1}}_q.
\end{equation}
Finally, by \eqref{Kab} and \eqref{Kinfb}, this is equal to
\[
\abs{\Sur_p(\la,\mu)}
=q^{-n(\la)-n(\mu)-\abs{\la}} b_{\la}(q)\,
\frac{K(\infty,\la'_1-\mu'_1)}{(q;q^2)_{\infty}}
\prod_{i\geq 1} \frac{K(\la'_i-\mu'_{i+1},\la'_{i+1}-\mu'_{i+1})}{(q;q^2)_{\ceil{\frac{m_i(\la)}{2}}}}.
\]
Combined with the definition of $\msympinf$ given in \eqref{symforminf} and the simple relation
\[
\prod_{i\geq 1} (q;q^2)_{\ceil{\frac{m_i(\la)}{2}}} (q^2;q^2)_{\floor{m_i(\la)/2}}=\prod_{i\geq 1} (q;q)_{m_i(\la)}=b_{\la}(q),
\]
this implies the claim.
\end{proof}

\begin{lemma}\label{Lem: Sum K(a,b) q^-kb}
For nonnegative integers $a$ and $k$,
\begin{equation}\label{Kabsumk}
\sum_{b=0}^a K(a,b)q^{-kb}
=q^{-\binom{k}{2}}(-1;q)_k
\sum_{l=0}^k \frac{(q^{-k};q)_l}{(q,-1;q)_l}\,q^{(a+1)l} .
\end{equation}
\end{lemma}

We defer the proof to Section~\ref{Subsec: Technical Lemmas K(a,b)}.
As before, by the dominated convergence theorem we obtain the following result in the $a\to\infty$ limit.

\begin{cor}
For $k$ a nonnegative integer,
\begin{equation}\label{Kinfbsumk}
\sum_{b=0}^{\infty} K(\infty,b)q^{-kb}=q^{-\binom{k}{2}}(-1;q)_k.
\end{equation}
\end{cor}

We are now ready to prove the next theorem.

\begin{thm}\label{Thm: moments of sym k1 dots km not normalized}
Let $\mu,\kappa$ be partitions and $r:=l(\kappa)$.
Then
\[
M_{p,\mu}(\mktilda)
=q^{-\sum_{i\geq 1} \binom{\mu'_i+\kappa_i}{2}}
\bigg(\prod_{i=1}^r (-1;q)_{\kappa_i}\bigg)
\mathcal{F}^{\kappa}\big(q^{m_1(\mu)+1},\dots,q^{m_{r-1}(\mu)+1}\big).
\]
\end{thm}

\begin{proof}
Recall that $K(a,b)$ is defined in \eqref{Kab} for all integers $a,b$ and vanishes unless $0\leq b\leq a$.
By \eqref{mumoment}, \eqref{mktilda} and Lemma~\ref{Lem: PSym inf times sur product},
\[
M_{p,\mu}(\mktilda)=\sum_{\la} 
q^{-n(\mu)-\sum_{i=1}^r \kappa_i\la'_i} 
K\big(\infty,\la'_1-\mu'_1\big)
\prod_{i\geq 1} K\big(\la'_i-\mu'_{i+1},\la'_{i+1}-\mu'_{i+1}\big).
\]
Since we have seen that \eqref{Kinfbsumk} follows from \eqref{Kabsumk} by taking the $a\to\infty$ limit and appealing to the dominated convergence theorem, we may in the following first evaluate
\[
A_{N;\mu,\kappa}:=\sum_{\la} 
q^{-n(\mu)-\sum_{i=1}^r \kappa_i\la'_i} 
K\big(N-\mu'_1,\la'_1-\mu'_1\big)
\prod_{i\geq 1} K\big(\la'_i-\mu'_{i+1},\la'_{i+1}-\mu'_{i+1}\big),
\]
for $N$ a nonnegative integer such that $N\geq l(\mu)$, and then take the
$N\to\infty$ limit to obtain an expression for $M_{p,\mu}(\mktilda)$.
Replacing $\la'_i\mapsto \nu_i+\mu'_i$ in the definition of $A_{N;\mu,\kappa}$ and using that $\mu'_i-\mu'_{i+1}=m_i(\mu)$ yields
\begin{align*}
A_{N;\mu,\kappa}&=
\sum_{\nu} q^{-n(\mu)-\sum_{i=1}^r \kappa_i(\nu_i+\mu'_i)} 
\prod_{i\geq 1} K\big(\nu_{i-1}+m_{i-1}(\mu),\nu_i\big) \\
&= q^{-n(\mu)-\sum_{i=1}^r \kappa_i\mu'_i}
\sum_{\nu}  
\prod_{i=1}^r q^{-\kappa_i \nu_i} 
K\big(\nu_{i-1}+m_{i-1}(\mu),\nu_i\big) 
\prod_{i\geq r} K\big(\nu_i+m_i(\mu),\nu_{i+1}\big),
\end{align*}
where $\nu_0+m_0(\mu):=N-\mu'_1$.
By Lemma~\ref{Kabsum} this simplifies to
\[
A_{N;\mu,\kappa}=q^{-n(\mu)-\sum_{i=1}^r \kappa_i\mu'_i} 
\sum_{\nu_1,\dots,\nu_r\geq 0} \prod_{i=1}^r q^{-\kappa_i\nu_i} 
K\big(\nu_{i-1}+m_{i-1}(\mu),\nu_i\big). 
\]
Next we claim that for $1\leq s\leq r+1$ and $\nu_0,\dots,\nu_{s-1},m_0,\dots,m_{r-1}\in\mathbb{Z}_{\geq 0}$,
\begin{align}\label{induction}
&\sum_{\nu_s,\dots,\nu_r\geq 0} 
\prod_{i=1}^r q^{-\kappa_i\nu_i} K\big(\nu_{i-1}+m_{i-1},\nu_i\big)  \\
&\quad= \prod_{i=1}^{s-1} q^{-\kappa_i\nu_i} K\big(\nu_{i-1}+m_{i-1},\nu_i\big) 
\notag \\
&\qquad\times
\sum_{l_s=0}^{\kappa_s}\cdots\sum_{l_r=0}^{\kappa_r} q^{l_s\nu_{s-1}}
\prod_{i=s}^r \frac{q^{-\binom{\kappa_i}{2}+(m_{i-1}+1)l_i} (-1;q)_{\kappa_i} (q^{-\kappa_i};q)_{l_i+l_{i+1}}}
{(q,-1;q)_{l_i}(-q^{1-\kappa_i},q^{-\kappa_i};q)_{l_{i+1}}}, \notag
\end{align}
where $l_{r+1}:=0$.
Clearly this is true for $s=r+1$.
Now assume \eqref{induction} is true for some $s>1$.
Then, summing both sides of \eqref{induction} over $\nu_{s-1}$ yields
\begin{align*}
&\sum_{\nu_{s-1},\dots,\nu_r\geq 0} 
\prod_{i=1}^r q^{-\kappa_i\nu_i} K\big(\nu_{i-1}+m_{i-1},\nu_i\big)  \\
&\quad= \prod_{i=1}^{s-2} q^{-\kappa_i\nu_i} K\big(\nu_{i-1}+m_{i-1},\nu_i\big)  \\
&\qquad\times
\sum_{l_s=0}^{\kappa_s}\cdots\sum_{l_r=0}^{\kappa_r} \bigg(
\sum_{\nu_{s-1}\geq 0} q^{-(\kappa_{s-1}-l_s)\nu_{s-1}} 
K\big(\nu_{s-2}+m_{s-2},\nu_{s-1}\big) \\
&\qquad\qquad\qquad\qquad\times
\prod_{i=s}^r \frac{q^{-\binom{\kappa_i}{2}+(m_{i-1}+1)l_i} (-1;q)_{\kappa_i} (q^{-\kappa_i};q)_{l_i+l_{i+1}}}
{(q,-1;q)_{l_i}(-q^{1-\kappa_i},q^{-\kappa_i};q)_{l_{i+1}}}\bigg).
\end{align*}
Since $l_s$ is summed from $0$ to $\kappa_s$ and since $\kappa$ is a partition, we have $\kappa_{s-1}-l_s\geq 0$. Hence we may apply 
Lemma~\ref{Lem: Sum K(a,b) q^-kb} for $k=\kappa_{s-1}-l_s$ to rewrite the sum over $\nu_{s-1}$ as
\begin{align*}
&q^{-\binom{\kappa_{s-1}-l_s}{2}}(-1;q)_{\kappa_{s-1}-l_s}
\sum_{l_{s-1}=0}^{\kappa_{s-1}-l_s} \frac{(q^{-(\kappa_{s-1}-l_s)};q)_{l_{s-1}}}{(q,-1;q)_{l_{s-1}}}\,q^{(\nu_{s-2}+m_{s-2}+1)l_{s-1}} \\
&\quad=
\sum_{l_{s-1}=0}^{\kappa_{s-1}} 
\frac{q^{-\binom{\kappa_{s-1}}{2}+(\nu_{s-2}+m_{s-2}+1)l_{s-1}}
(-1;q)_{\kappa_{s-1}}(q^{-\kappa_{s-1}};q)_{l_{s-1}+l_s}}
{(q,-1;q)_{l_{s-1}}(-q^{1-\kappa_{s-1}},q^{-\kappa_{s-1}};q)_{l_s}},
\end{align*}
where in the expression on the right we have changed the upper bound in the sum over $l_{s-1}$ from $\kappa_{s-1}-l_s$ to $\kappa_{s-1}$ using the fact that the summand vanishes unless $l_{s-1}+l_s\leq\kappa_{s-1}$.
Substituting the above into the preceding equation results in
\begin{align*}
&\sum_{\nu_{s-1},\dots,\nu_r\geq 0} 
\prod_{i=1}^r q^{-\kappa_i\nu_i} K\big(\nu_{i-1}+m_{i-1},\nu_i\big)  \\
&\qquad= \prod_{i=1}^{s-2} q^{-\kappa_i\nu_i} K\big(\nu_{i-1}+m_{i-1},\nu_i\big) \\
&\qquad\quad\times
\sum_{l_{s-1}=0}^{\kappa_{s-1}}\cdots\sum_{l_r=0}^{\kappa_r} 
q^{l_{s-1}\nu_{s-2}}
\prod_{i=s-1}^r \frac{q^{-\binom{\kappa_i}{2}+(m_{i-1}+1)l_i} (-1;q)_{\kappa_i} (q^{-\kappa_i};q)_{l_i+l_{i+1}}}
{(q,-1;q)_{l_i}(-q^{1-\kappa_i},q^{-\kappa_i};q)_{l_{i+1}}}.
\end{align*}
Since this is \eqref{induction} with $s$ replaced by $s-1$, this proves \eqref{induction} for all $1\leq s\leq r+1$. 
Taking $s=1$ and $m_i=m_i(\mu)$, this implies that
\begin{align*}
A_{N;\mu,\kappa}=q^{-\sum_{i\geq 1}\binom{\mu'_i+\kappa_i}{2}} 
\sum_{l_1=0}^{\kappa_1}\cdots\sum_{l_r=0}^{\kappa_r} q^{l_1\nu_0}
\prod_{i=1}^r \frac{q^{(m_{i-1}(\mu)+1)l_i} (-1;q)_{\kappa_i} (q^{-\kappa_i};q)_{l_i+l_{i+1}}}
{(q,-1;q)_{l_i}(-q^{1-\kappa_i},q^{-\kappa_i};q)_{l_{i+1}}},
\end{align*}
where $\nu_0+m_0(\mu)=N-\mu'_1$ as before. 
In the large $N$ limit, the summand vanishes unless $l_1=0$, leading to
\begin{align*}
M_{p,\mu}(\mktilda)&=\lim_{N\to\infty} A_{N;\mu,\kappa} \\
&=q^{-\sum_{i\geq 1}\binom{\mu'_i+\kappa_i}{2}} 
\sum_{l_2=0}^{\kappa_2}\cdots\sum_{l_r=0}^{\kappa_r} 
\prod_{i=1}^r \frac{q^{(m_{i-1}(\mu)+1)l_i} (-1;q)_{\kappa_i} (q^{-\kappa_i};q)_{l_i+l_{i+1}}}
{(q,-1;q)_{l_i}(-q^{1-\kappa_i},q^{-\kappa_i};q)_{l_{i+1}}} \\
&=q^{-\sum_{i\geq 1}\binom{\mu'_i+\kappa_i}{2}} 
\sum_{l_2=0}^{\kappa_2}\cdots\sum_{l_r=0}^{\kappa_r} 
\prod_{i=1}^r \frac{q^{(m_{i-1}(\mu)+1)l_i} (-1;q)_{\kappa_i} (q^{-\kappa_i};q)_{l_i+l_{i+1}}}
{(q,-1,-q^{1-\kappa_{i-1}},q^{-\kappa_{i-1}};q)_{l_i}},
\end{align*}
where $l_1=l_{r+1}:=0$.
Recalling definition \eqref{SD} completes the proof.
\end{proof}

\begin{cor}\label{Cor: bound on moments}
Let $\kappa$ be a partition of length $r$ and $\mu$ an arbitrary partition.
Then there exists a constant $C=C_{\kappa}$ independent of $\mu$ such that
\[
M_{p,\mu}\big(\mktilda\big) \leq C p^{n(\mu)+\abs{\kappa}l(\mu)}.
\]
\end{cor}

This corollary proves the first claim of Theorem~\ref{Thm: Same moments S odd size}.

\begin{proof}
Since $p>1$ (and thus $q=1/p<1$),
\begin{align*}
\left|\mathcal{F}^{\kappa}\big(q^{m_1(\mu)+1},\dots,q^{m_{r-1}(\mu)+1}\big)\right|
&=\left|
\sum_{l_2=0}^{\kappa_2}\cdots\sum_{l_r=0}^{\kappa_r}
\prod_{i=1}^r
\frac{(q^{-\kappa_i};q)_{l_i+l_{i+1}}} {(q,-1,-q^{1-\kappa_{i-1}},q^{-\kappa_{i-1}};q)_{l_i}}\,q^{(m_{i-1}(\mu)+1)l_i}
\right|\\
&\leq 
\sum_{l_2=0}^{\kappa_2}\cdots\sum_{l_r=0}^{\kappa_r}
\bigg\vert
\prod_{i=1}^r
\frac{(q^{-\kappa_i};q)_{l_i+l_{i+1}}} {(q,-1,-q^{1-\kappa_{i-1}},q^{-\kappa_{i-1}};q)_{l_i}}\bigg\vert,
\end{align*}
which is independent of $\mu$.
Furthermore,
\begin{align*}
\sum_{i\geq 1} \binom{\mu'_i+\kappa_i}{2}&=
n(\mu)+n(\kappa')+\sum_{i\geq 1} \mu_i'\kappa_i \\
&\leq n(\mu)+n(\kappa')+\sum_{i,j\geq 1} \mu'_i\kappa_j
=n(\mu)+n(\kappa')+\abs{\kappa} l(\mu)
\end{align*}
and thus 
\[
p^{\sum_{i\geq 1} \binom{\mu'_i+\kappa_i}{2}}\leq p^{n(\mu)+n(\kappa')+\abs{\kappa} l(\mu)}.
\]
The result now follows from Theorem~\ref{Thm: moments of sym k1 dots km not normalized}.
\end{proof}

Next, we will show that for $\kappa$ a strict partition of length $r$ and $S\in\rodd$, the measures $\mkStilda$ have the same moments as $\mktilda$.
For this, we need to evaluate sums of the form $\sum_b (-1)^b K(a,b) q^{-kb}$.

\begin{lemma}\label{minusonesum}
For a nonnegative integer $a$ and a positive integer $k$,
\[
\sum_{b=0}^a (-1)^b K(a,b) q^{-kb}
=(-1)^a q^{a+1-\binom{k+1}{2}} (-q^2;q)_{k-1} 
\sum_{j=0}^{k-1} \frac{(-1)^j q^{(a-k+2)j+\binom{j}{2}}}{(-q^2;q)_j}
\qbin{k-1}{j}_q.
\]
\end{lemma}

We defer the proof to Section~\ref{Subsec: Technical Lemmas K(a,b)}.
Once again the dominated convergence theorem implies the $a\to\infty$ limiting case.

\begin{cor}\label{Cor: Alt sum K(infty,b) q^-kb}
For $k$ a positive integer,
\begin{equation}\label{zero}
\sum_{b=0}^{\infty}(-1)^b K(\infty,b)q^{-kb}=0.
\end{equation}
\end{cor}

We now have the tools to show that $\mkStilda$ for $\kappa$ and $S$ as above has the same moments as $\mktilda$, thus proving the second part of Theorem~\ref{Thm: Same moments S odd size}.  
Throughout the rest of the paper, we use $\chi(\cdot)$ to denote an indicator function where $\chi(\text{true}) = 1$ and $\chi(\text{false}) = 0$.

\begin{thm}\label{Thm: Same moments S}
Let $\kappa,\mu$ be partitions such that $\kappa$ is strict.
Set $r:=l(\kappa)$. 
Then, for $S\in\rodd$,
\[
M_{p,\mu}\big(\mkStilda\big)=M_{p,\mu}\big(\mktilda\big).
\]
\end{thm}

\begin{proof}
From definition \eqref{signed} it follows that it suffices to prove that for every partition $\mu$, we have
\[
\sum_{\la} (-1)^{\sum_{i\in S}\la_i'} \mktilda(\la)  \abs{\Sur_p(\lambda, \mu)}
=0,
\]
provided $S\in\rodd$.
By \eqref{mktilda} and Lemma~\ref{Lem: PSym inf times sur product}, we see that the left-hand side is equal to
\[
\sum_{\la} (-1)^{\sum_{i\in S}\la_i'} 
q^{-n(\mu)-\sum_{i=1}^r \kappa_i\la_i'} 
K\big(\infty,\la'_1-\mu'_1\big)
\prod_{i\geq 1} K\big(\la'_i-\mu'_{i+1},\la'_{i+1}-\mu'_{i+1}\big).
\]
By substituting $\nu_i=\la_i'-\mu_i'$ for $i\geq 1$ and defining $\nu_0+m_0(\mu)=\infty$, this simplifies to
\begin{align*}
&
\sum_{\nu} 
(-1)^{\sum_{i\in S}(\nu_i+\mu_i')} 
q^{-n(\mu)-\sum_{i=1}^r \kappa_i(\nu_i+\mu'_i)} 
\prod_{i\geq 1} K\big(\nu_{i-1}+m_{i-1}(\mu),\nu_i\big) \\
&= (-1)^{\sum_{i\in S}\mu_i'}
q^{-n(\mu)-\sum_{i=1}^r \kappa_i\mu'_i}
\sum_{\nu}  
\prod_{i=1}^r (-1)^{\chi(i\in S)\nu_i} q^{-\kappa_i \nu_i} 
K\big(\nu_{i-1}+m_{i-1}(\mu),\nu_i\big) 
\prod_{i\geq r} K\big(\nu_i+m_i(\mu),\nu_{i+1}\big).
\end{align*}
By Lemma~\ref{Kabsum}, this can be simplified as
\[
(-1)^{\sum_{i\in S}\mu_i'}
q^{-n(\mu)-\sum_{i=1}^r \kappa_i\mu'_i}
\sum_{\nu_1,\dots,\nu_r\geq 0}  
\prod_{i=1}^r (-1)^{\chi(i\in S)\nu_i} q^{-\kappa_i \nu_i} 
K\big(\nu_{i-1}+m_{i-1}(\mu),\nu_i\big).
\]
Let us denote
\[
B_{\kappa,\mu,S}
=\sum_{\nu_1,\dots,\nu_r\geq 0}  
\prod_{i=1}^r (-1)^{\chi(i\in S)\nu_i} q^{-\kappa_i \nu_i} 
K\big(\nu_{i-1}+m_{i-1}(\mu),\nu_i\big).
\]
Therefore, it is enough to prove that $B_{\kappa,\mu,S}=0$ whenever $\kappa$ is a strict partition of length $r$ and $S\in\rodd$.
We prove this by induction on $r$.

The base case is $r=1$. Since the set $S$ has odd size, the only possibility is $S=\{1\}$. Therefore, the base case is that for $\kappa$ a strict partition of length $1$,
\[
B_{\kappa,\mu,S}
=\sum_{\nu_1=0}^{\infty} K(\infty,\nu_1) (-1)^{\nu_1} q^{-\kappa_1\nu_1}
=0.
\]
This follows from Corollary~\ref{Cor: Alt sum K(infty,b) q^-kb}.

Next, assume that the claim is true for some $r$. We will prove it for $r+1$. Let $\kappa$ be a strict partition of length $r+1$ and $S\subseteq[r+1]$ be a subset of odd size.
We divide the proof into two cases.
\begin{enumerate}
\item Case 1: $r+1\in S$. 
Denote
\[
T=
\begin{cases}
(S\cup\{r\})\setminus\{r+1\} &\text{if }r\notin S\\
S\setminus\{r,r+1\} &\text{if }r\in S.
\end{cases}
\]
Note in particular that $T\subseteq [r]$ is a set of odd cardinality. 
For $0\leq j\leq \kappa_{r+1}-1$, let $\tau^{(j)}$ be the partition $(\kappa_1,\dots,\kappa_{r-1},\kappa_{r}-1-j)$. Since $\kappa$ is a strict partition of length $r+1$ and $j+1\leq \kappa_{r+1}$, $\tau^{(j)}$ is a strict partition of length $r$.

By applying Lemma~\ref{minusonesum} with $a=\nu_r+m_r(\mu)$ and $k=\kappa_{r+1}$, we obtain
\[
\sum_{\nu_{r+1}=0}^{\nu_r+m_r(\mu)} 
(-1)^{\nu_{r+1}} 
K(\nu_r+m_r(\mu),\nu_{r+1}) 
q^{-\kappa_{r+1}\nu_{r+1}}
=\sum_{j=0}^{\kappa_{r+1}-1}
(-1)^{\nu_r}
q^{(j+1)\nu_r} \varphi_{j;m_{r}(\mu),\kappa_{r+1}},
\]
where
\[
\varphi_{j;m,k}
=(-1)^{m+j} q^{m+1-\binom{k+1}{2} +(m-k+2)j+\binom{j}{2} } 
\frac{ (-q^2;q)_{k-1}}{(-q^2;q)_j}
\qbin{k-1}{j}_q.
\]
Therefore,
\[
B_{\kappa,\mu,S}
=
\sum_{j=0}^{\kappa_{r+1}-1}
\varphi_{j;m_{r}(\mu),\kappa_{r+1}} B_{\tau^{(j)},\mu,T}.
\]
By the induction hypothesis, $B_{\tau^{(j)},\mu,T}=0$ for all $\tau^{(j)}$. We conclude that $B_{\kappa,\mu,S}=0$.

\item Case 2: $r+1\notin S$. So $S\subseteq [r]$ is a set of odd cardinality. 
For $0\leq j\leq \kappa_{r+1}$, let $\tau^{(j)}$ be the partition $(\kappa_1,\dots,\kappa_{r-1},\kappa_{r}-j)$. Since $\kappa$ is a strict partition of length $r+1$, $\tau^{(j)}$ is a strict partition of length $r$.

By applying Lemma~\ref{Lem: Sum K(a,b) q^-kb} with $a=\nu_r+m_r(\mu)$ and $k=\kappa_{r+1}$, we obtain
\[
\sum_{\nu_{r+1}=0}^{\nu_r+m_r(\mu)} K(\nu_r+m_r(\mu),\nu_{r+1})q^{-\kappa_{r+1}\nu_{r+1}}
= \sum_{j=0}^{\kappa_{r+1}} q^{\nu_r j} \psi_{j;m_r(\mu), \kappa_{r+1}},
\]
where
\[
\psi_{j;m,k}
=q^{-\binom{k}{2}+(m+1)j}
(-1;q)_k
\frac{(q^{-k};q)_j}{(q,-1;q)_j}.
\]
Therefore,
\[
B_{\kappa,\mu,S}
=
\sum_{j=0}^{\kappa_{r+1}}
\psi_{j;m_{r}(\mu),\kappa_{r+1}} B_{\tau^{(j)},\mu,S}.
\]
By the induction hypothesis, $B_{\tau^{(j)},\mu,S}=0$ for all $\tau^{(j)}$, so that $B_{\kappa,\mu,S}=0$.
\end{enumerate}
Since we have completed the induction step for $r+1\in S$ and $r+1\notin S$, this completes the proof.
\end{proof}

To conclude this section we prove Theorem~\ref{Thm: Same moments alpha family}. In this proof we apply Corollary~\ref{Cor: H is injectiev}, which we will prove in Section~\ref{sec:shapeA_r}.

\begin{proof}[Proof of Theorem~\ref{Thm: Same moments alpha family}]
First we show that for $\ba \in \A_r$, we have $\mkAtilda(\la) \ge 0$ for all $\la$. 
By \eqref{signed} and \eqref{mkAtilda},
\[
\mkAtilda(\la)
=\bigg(1+\sum_{S\in\rodd} (-1)^{\sum_{i\in S}\la'_i} \,\ba(S) \bigg) \mktilda(\la).
\]
Since
\[
(-1)^{\sum_{i\in S}\la'_i}=(-1)^{\sum_{i\in S}\chi(\la'_i \text{ odd})}=
(-1)^{\abs{S}+\sum_{i\in S}\chi(\la'_i \text{ even})}=
-(-1)^{\sum_{i\in S}\chi(\la'_i \text{ even})},
\]
this implies
\[
\mkAtilda(\la)
=\bigg(1-\sum_{S\in\rodd} (-1)^{\abs{S\cap T}} \,\ba(S) \bigg) \mktilda(\la),
\]
where $T:=\{1\leq i\leq r:~\text{ $\la_i'$ is even}\}$.
By the definition of $\A_r$, the right-hand side is nonnegative.
Next, to obtain \eqref{noalphadependence}, we proceed as follows:
\begin{align*}
M_{p,\mu}\big(\mkAtilda\big)
&= \sum_{S\in \rodd} \ba(S) M_{p,\mu}\big(\mkStilda\big)
+\Big(1-\sum_{S\in \rodd} \ba(S)\Big) M_{p,\mu}\big(\mktilda\big) \\
&=M_{p,\mu}(\mktilda),
\end{align*}
where the second equality follows from Theorem~\ref{Thm: Same moments S}.

Finally, suppose that $\mkAtilda=\mkBtilda$ for $\ba,\bb\in \A_r$.
Since $\mktilda(\la)>0$ for every partition $\la$, it follows that for every subset $T\subseteq [r]$,
\[
\sum_{S\in\rodd} (-1)^{\abs{S\cap T}} \ba(S)
=\sum_{S\in\rodd} (-1)^{\abs{S\cap T}} \bb(S).
\]
By Corollary~\ref{Cor: H is injectiev}, this implies $\ba=\bb$.
\end{proof}

\subsection{Normalized measures}

The measures $\mktilda$ and $\mkAtilda$ have not been normalized to sum to $1$ over all partitions.  In this section, we define the appropriate normalization constants to turn them into distributions.  We then compute the moments of these normalized distributions, which allows us to deduce Theorem~\ref{Thm_r=1_one-par-family}.  We also highlight the case where $l(\kappa) = 2$.

For a partition $\kappa$, we define the distribution on partitions given by
\[
\mkNormal
:=q^{n(\kappa')}\bigg(\prod_{i\geq 1} 
\frac{1}{(-1;q)_{\kappa_i}}\bigg)
\frac{\mktilda}{\mathcal{F}^{\kappa}(q,\dots,q)}.
\]
Analogously, for $\kappa$ strict of length $r$,
\[
\mkANormal
:= q^{n(\kappa')}\bigg(\prod_{i=1}^r 
\frac{1}{(-1;q)_{\kappa_i}}\bigg)
\frac{\mkAtilda}{\mathcal{F}^{\kappa}(q,\dots,q)},
\]
where $\ba\in\A_r$.

It follows from Theorem~\ref{Thm: moments of sym k1 dots km not normalized} and Theorem~\ref{Thm: Same moments alpha family} that for every $\ba\in\A_r$,
\[
M_{p,\mu}(\mkANormal)=M_{p,\mu}(\mkNormal)
=q^{-n(\mu)-\sum_{i\geq 1}\mu'_i\kappa_i}
\frac{\mathcal{F}^{\kappa}(q^{m_1(\mu)+1},\dots,q^{m_{r-1}(\mu)+1})}{\mathcal{F}^{\kappa}(q,\dots,q)}.
\]
For $\mu=0$ this implies that
\[
\sum_{\la} \mkANormal(\la)=\sum_{\la} \mkNormal(\la)=1,
\]
so that $\mkNormal$ and $\mkANormal$ are distributions.
In particular, for $r=1$, we have $\mathcal{F}^{(k)}=1$ so that
\[
\spmk=\frac{q^{\binom{k}{2}}}{(-1;q)_k} \,\tilde{P}^{\Sym,(k)}_p
\]
and
\[
M_{p,\mu}\big(\qpmk\big)
=M_{p,\mu}\big(\spmk\big)
=q^{-n(\mu)-k l(\mu)}=p^{n(\mu)+k l(\mu)},
\]
establishing Theorem~\ref{Thm_r=1_one-par-family}.
Similarly, for $r=2$, we obtain
\[
P^{\Sym,\kappa}_p
=\frac{q^{n(\kappa')} \tilde{P}^{\Sym,\kappa}_p}
{(-1;q)_{\kappa_1}(-1;q)_{\kappa_2} 
\,{_3\phi_2}(0,0,q^{-\kappa_2};-1,-q^{1-\kappa_1};q,q)}\,
\]
and
\[
M_{p,\mu}\big(P^{\Sym,\kappa}_{p,\ba}\big)=
M_{p,\mu}\big(P^{\Sym,\kappa}_p\big)
=q^{-n(\mu)-\kappa_1\mu_1'-\kappa_2\mu_2'}\,
\frac{{_3\phi_2}(0,0,q^{-\kappa_2};-1,-q^{1-\kappa_1};q,q^{\mu'_1-\mu'_2+1})}
{{_3\phi_2}(0,0,q^{-\kappa_2};-1,-q^{1-\kappa_1};q,q)},
\]
where $\ba\in\A_2$.

\section{Moments of other measures}\label{Sec: Moments of other measures}

\subsection{Moments of the standard Cohen--Lenstra distribution}\label{newold}

Let $d$ be a nonnegative integer and $u$ a real number such that $0<u<p$.
In this section, we compute the $\mu$-moments of the distribution $P_{d,u}$, providing a new proof of Theorem~\ref{thm53Fulman}.

We start by defining transition weights that factorize the distribution $P_{d,u}$. 
For $a,b$ integers and $w>0$, define
\begin{equation}\label{Eqn: def of L(a,b)}
L_w(a,b):=w^b q^{b^2} \frac{(wq;q)_a}{(wq;q)_b}\qbin{a}{b}_q.
\end{equation}
We note that $L_w(a,b)$ vanishes unless $0\leq b\leq a$.
Similarly, define
\[
L_w(\infty,b):=\lim_{a\to\infty} L_w(a,b)=w^b q^{b^2}
\frac{(wq;q)_{\infty}}{(q,wq;q)_b}.
\]

In analogy with Lemma~\ref{Kabsum}, we show that 
$\sum_b L_w(a,b)=1$.

\begin{lemma}\label{LemLsum1}
For $a$ a nonnegative integer,
\[
\sum_{b=0}^a L_w(a,b)=1.
\]
\end{lemma}

\begin{proof}
Using basic hypergeometric notation,
\[
\sum_{b=0}^a L_w(a,b)=(wq;q)_a\:\qhyp{1}{1}{q^{-a}}{wq}{q,wq^{a+1}}=1,
\]
where the last equality follows from \eqref{Eqn: 1phi1}.
\end{proof}

\begin{cor}
We have
\[
\sum_{b=0}^{\infty} L_w(\infty,b)=1.
\]
\end{cor}

\begin{proof}
This follows from Lemma~\ref{LemLsum1} and an appeal to the dominated convergence theorem.
\end{proof}

Next, we record the transition factorization of the distribution $P_{d,u}$. 

\begin{prop}\label{proposition}
For $\la$ a partition,
\[
P_{d,u}(\la)=L_u(d,\la_1') \prod_{i\geq 1} L_u(\la_i',\la_{i+1}').
\]
Therefore, with respect to the distribution $P_{d,u}$ on partitions,
\[
\P(\la'_{i+1}=b\mid \la'_i=a)=L_u(a,b)\quad
\text{and}\text\quad
\P(\la_1'=b)=L_u(d,b).
\]
\end{prop}

\begin{proof}
Let $\la=(\la_1,\la_2,\dots)$ be a partition. 
Further let $\la'_0:=d$.
Then, by \eqref{Eqn: def of L(a,b)},
\[
\prod_{i\geq 0} L_u(\la_i',\la_{i+1}')
=\prod_{i\geq 0} u^{\la_{i+1}'} q^{\la_{i+1}'^2} 
\frac{(uq;q)_{\la_i'}}{(uq;q)_{\la_{i+1}'}}\qbin{\la_i'}{\la_{i+1}'}_q\\
=\frac{u^{\abs{\la}} q^{2n(\la)+\abs{\la}}(q,uq;q)_d}{(q;q)_{d-l(\la)}b_{\la}(q)}
=P_{d,u}(\lambda). \qedhere
\]
\end{proof}

\begin{proof}[Proof of Theorem~\ref{thm53Fulman}]
By \eqref{Sur1} and the homogeneity of the Hall--Littlewood symmetric functions,
\[
\abs{\Sur_p(\la,\mu)}=q^{-n(\mu)-n(\la)-\abs{\la}} b_{\la}(q) 
u^{\abs{\mu}-\abs{\la}} P_{\la/\mu}(uq,uq^2,\dots;q).
\]
Therefore, the $\mu$-moment of $P_{d,u}$ is given by
\begin{align*}
M_{p,\mu}(P_{d,u})
&=\sum_{\la} P_{d,u}(\la) \abs{\Sur_p(\la,\mu)} \\
&=u^{\abs{\mu}} q^{-n(\mu)} (uq;q)_d 
\sum_{\la} b_{\la}(q) P_{\la}(1,q,\dots,q^{d-1};q)
P_{\la/\mu}(uq,uq^2,\dots;q)\\
&=u^{\abs{\mu}} q^{-n(\mu)} (uq;q)_d 
\sum_{\la}  Q_{\la}(1,q,\dots,q^{d-1};q)
P_{\la/\mu}(uq,uq^2,\dots;q).
\end{align*}
By the skew Cauchy identity \eqref{Eqn: Skew Cauchy identity} with $x_i=q^{i-1}$ for $1\leq i\leq d$, $x_i=0$ for $i>d$ and $y_j=uq^j$ for $j\geq 1$, this simplifies to
\[
M_{p,\mu}(P_{d,u})
=u^{\abs{\mu}} q^{-n(\mu)} Q_{\mu}(1,q,\dots,q^{d-1};q).
\]
Finally, by \eqref{Eq_p213} this yields
\[
M_{p,\mu}(P_{d,u}) =u^{\abs{\mu}} (q;q)_{l(\mu)} \qbin{d}{l(\mu)}_q. \qedhere
\]
\end{proof}

\subsection{Moments of a generalization of Malle's distribution}\label{Hallmome2}
In this subsection, we prove Theorem~\ref{thm: moments of Malle}.
First, recall that
\[
\mMalled(\lambda)
=P_{d,q^t}(\la)\, 
\frac{q^{-\binom{l(\la)}{2}}(q^{t+1};q)_{l(\la)}}
{(q^{2t+2};q^2)_d}.
\]
We note that by the $q$-binomial theorem \eqref{qbt} with $(z,n)\mapsto (-wzq,a)$, 
\begin{equation}\label{qbtone}
\sum_{b=0}^a L_w(a,b) \, \frac{z^b q^{-\binom{b}{2}}(wq;q)_b}{(wq,-wzq;q)_a}=1.
\end{equation}
The dominated convergence theorem implies that
\[
\sum_{b=0}^{\infty} L_w(\infty,b) \, \frac{z^b q^{-\binom{b}{2}}(wq;q)_b}{(wq,-wzq;q)_{\infty}}=1.
\]
Taking $z=1$ and using $(a,-a;q)_n=(a^2;q^2)_n$, this provides an algorithm for choosing a random partition according to $\mMalled$.

\begin{prop}
For $\la$ a partition,
\[
\mMalled(\la)
=\frac{q^{-\binom{\la_1'}{2}}(q^{t+1};q)_{\la_1'}}{(q^{2t+2};q^2)_{d}}\,
L_{q^t}(d,\la_1') \prod_{i\geq 1} L_{q^t}(\la_i',\la_{i+1}').
\]
With respect to the distribution $\mMalled$ on partitions,
\[
\P(\la_1'=b)=
\frac{q^{-\binom{b}{2}}(q^{t+1};q)_b}{(q^{2t+2};q^2)_{d}}\, 
L_{q^t}(d,b)
\quad \text{and}\text\quad
\P(\la'_{i+1}=b \mid \la'_i=a)=L_{q^t}(a,b).
\]
\end{prop}

We are now ready to prove Theorem~\ref{thm: moments of Malle}.

\begin{proof}[Proof of Theorem~\ref{thm: moments of Malle}]
By \eqref{nice} and \eqref{Sur2}, 
\[
P_{d,u}(\la) \abs{\Sur_p(\la,\mu)}\\
=u^{\abs{\la}}\,
q^{n(\la/\mu)+n(\la)-n(\mu)+\abs{\la}-\abs{\mu}}
\frac{(q,uq;q)_d}{(q;q)_{d-\la'_1}(q;q)_{\la'_1-\mu'_1}}
\prod_{i\geq 2} \qbin{\la'_{i-1}-\mu'_i}{\la'_i-\mu'_i}_q. 
\]
Since
\[
u^{\abs{\la}}\,
q^{n(\la/\mu)+n(\la)-n(\mu)+\abs{\la}-\abs{\mu}}(uq;q)_d
=u^{\abs{\mu}} \prod_{i\geq 1} (uq^{\mu'_i})^{\la'_i-\mu'_i} 
q^{(\la'_i-\mu'_i)^2} 
\frac{(uq^{\mu'_i+1};q)_{\la'_{i-1}-\mu'_i}}
{(uq^{\mu'_i+1};q)_{\la'_i-\mu'_i}},
\]
this yields
\begin{align}\label{Pdula}
&P_{d,u}(\la) \abs{\Sur_p(\la,\mu)}\\
&\qquad=u^{\abs{\mu}}\frac{(q;q)_d}{(q;q)_{d-\mu'_1}}
\prod_{i\geq 0} (uq^{\mu'_{i+1}})^{\la'_{i+1}-\mu'_{i+1}} 
q^{(\la'_{i+1}-\mu'_{i+1})^2} 
\frac{(uq^{\mu'_{i+1}+1};q)_{\la'_i-\mu'_{i+1}}}
{(uq^{\mu'_{i+1}+1};q)_{\la'_{i+1}-\mu'_{i+1}}}
\qbin{\la'_i-\mu'_{i+1}}{\la'_{i+1}-\mu'_{i+1}}_q \notag \\
&\qquad=u^{\abs{\mu}}
(q;q)_{l(\mu)} \qbin{d}{l(\mu)}
\prod_{i\geq 1} L_{uq^{\mu'_i}}\big(\la'_{i-1}-\mu'_i,\la'_i-\mu'_i\big),
\notag
\end{align}
where $\la'_0:=d$.
With $u=q^t$, this implies
\begin{align*}
&\mMalled(\lambda) \abs{\Sur_p(\la,\mu)}\\
&\qquad=\frac{q^{t\abs{\mu}-\binom{l(\la)}{2}}(q^{t+1};q)_{l(\la)}(q;q)_{l(\mu)}}
{(q^{t+1};q)_d(-q^{t+1};q)_{d}}
\qbin{d}{l(\mu)}_q
\prod_{i\geq 1} L_{q^{\mu'_i+t}}\big(\la'_{i-1}-\mu'_i,\la'_i-\mu'_i\big).
\end{align*}
For a positive integer $N$ such that $N\geq\mu_1$, we define 
\[
S_N:=\sum_{\mu\subseteq\la\subseteq (N^d)}
\frac{\mMalled(\la) \abs{\Sur_p(\la,\mu)}}{(q^{t+1};q)_{\la'_N}}.
\]
Then
\begin{align*}
S_N&=
\sum_{l_1,\dots,l_N\geq 0}\bigg(
\frac{q^{t\abs{\mu}-\binom{l_1+l(\mu)}{2}}
(q^{t+1};q)_{l_1+\mu'_1}(q;q)_{l(\mu)}}
{(q^{t+1};q)_d(-q^{t+1};q)_{d}} 
\qbin{d}{l(\mu)}_q
 \\
& \qquad\qquad\qquad\times L_{q^{l(\mu)+t}}\big(d-l(\mu),l_1\big)
\prod_{i=2}^N L_{q^{\mu'_i+t}}\big(l_{i-1}+m_{i-1}(\mu),l_i\big)\bigg)\\
&= 
\frac{q^{t\abs{\mu}-\binom{l(\mu)}{2}} (q;q)_{l(\mu)}}{ (-q^{d+t+1-l(\mu)};q)_{l(\mu)} }
\qbin{d}{l(\mu)}_q 
\sum_{l_1=0}^{d-l(\mu)}
\frac{q^{-\binom{l_1}{2}-l_1 l(\mu)}(q^{l(\mu)+t+1};q)_{l_1}}
{(q^{l(\mu)+t+1},-q^{t+1};q)_{d-l(\mu)}}
L_{q^{l(\mu)+t}}\big(d-l(\mu),l_1\big),
\end{align*}
where Lemma~\ref{LemLsum1} has been utilized to carry out the summations over $l_N,\dots,l_2$, in that order. 
The sum over $l_1$ once again is equal to $1$ by \eqref{qbtone} with 
$(w,z,a)\mapsto(q^{l(\mu)+t},q^{-l(\mu)},d-l(\mu))$, resulting in
\[
S_N=\frac{q^{t\abs{\mu}-\binom{l(\mu)}{2}} (q;q)_{l(\mu)}}{ (-q^{d+t+1-l(\mu)};q)_{l(\mu)} }
\qbin{d}{l(\mu)}_q.
\]
Letting $N$ tend to infinity shows that
\[
M_{p,\mu} (\mMalled)
=\frac{q^{t\abs{\mu}-\binom{l(\mu)}{2}} (q;q)_{l(\mu)}}{ (-q^{d+t+1-l(\mu)};q)_{l(\mu)} }
\qbin{d}{l(\mu)}_q.
\]
Finally, letting $d$ tend to infinity, we find
\[
M_{p,\mu} (\mMalle)
=q^{t\abs{\mu}-\binom{l(\mu)}{2}}. \qedhere
\]
\end{proof}

\subsection{Moments of a generalization of Garton's distribution}\label{Hallmome3}

In this section, we prove Theorem~\ref{thm: moments of Garton}. 

\begin{proof}[Proof of Theorem~\ref{thm: moments of Garton}]
Recall that
\[
\mGartond(\la)
=P_{d,1}(\la)\,
\frac{q^{-\sum_{i=1}^n\binom{\la'_i}{2}}
(q;q)_{\la'_n}}{(q;q)_d}
(q;q^2)_{\ceil{(d-l(\la))/2}}
\prod_{i=1}^{n-1} (q;q^2)_{\ceil{m_i(\la)/2}}.
\]
Then, by \eqref{Pdula},
\begin{align*}
&\mGartond(\la) \abs{\Sur_p(\la,\mu)}\\
&\qquad=(q;q)_{l(\mu)} \qbin{d}{l(\mu)}
\frac{q^{-\sum_{i=1}^n\binom{\la'_i}{2}} (q;q)_{\la'_n}}{(q;q)_d}
\prod_{i=0}^{n-1} (q;q^2)_{\ceil{m_i(\la)/2}}
\prod_{i\geq 1} L_{q^{\mu'_i}}\big(\la'_{i-1}-\mu'_i,\la'_i-\mu'_i\big),
\end{align*}
where $\la'_0:=d$ and $m_0(\la):=d-l(\la)$.
Now assume that $N$ is an integer such that $\mu_1,n\leq N$ and consider the sum
\[
S_N:=\sum_{\mu\subseteq\la\subseteq (N^d)}
\frac{\mGartond(\la) \abs{\Sur_p(\la,\mu)}}{(q;q)_{\la'_N}}.
\]
Once again applying Lemma~\ref{LemLsum1} yields
\begin{align*}
S_N&=(q;q)_{l(\mu)} \qbin{d}{l(\mu)}
\sum_{l_1,\dots,l_N\geq 0}\bigg(
\frac{q^{-\sum_{i=1}^n\binom{l_i+\mu'_i}{2}}(q;q)_{l_n+\mu'_n}
(q;q^2)_{\ceil{(d-l(\mu)-l_1)/2}}}{(q;q)_d}\\
&\quad\times
L_{q^{l(\mu)}}\big(d-l(\mu),l_1\big)
\prod_{i=1}^{n-1} (q;q^2)_{\ceil{(l_i-l_{i+1}+m_i(\mu))/2}}
\prod_{i=2}^N L_{q^{\mu'_i}}\big(l_{i-1}+m_{i-1}(\mu),l_i\big)\bigg) \\
&=(q;q)_{l(\mu)} \qbin{d}{l(\mu)}
\sum_{l_1,\dots,l_n\geq 0}\bigg(
\frac{q^{-\sum_{i=1}^n\binom{l_i+\mu'_i}{2}}(q;q)_{l_n+\mu'_n}
(q;q^2)_{\ceil{(d-l(\mu)-l_1)/2}}}{(q;q)_d}\\
&\quad\times
L_{q^{l(\mu)}}\big(d-l(\mu),l_1\big)
\prod_{i=2}^n (q;q^2)_{\ceil{(l_{i-1}-l_i+m_{i-1}(\mu))/2}}
L_{q^{\mu'_i}}\big(l_{i-1}+m_{i-1}(\mu),l_i\big)\bigg).
\end{align*}
Recalling the definition \eqref{Kab} of $K(a,b)$, we now observe that
\[
(q;q^2)_{\ceil{(a-b)/2}} L_{q^c}(a,b)=
q^{\binom{b}{2}+bc}\,\frac{(q;q)_{a+c}}{(q;q)_{b+c}} K(a,b).
\]
Hence the above may be rewritten as
\[
S_N=q^{-\sum_{i=1}^n\binom{\mu'_i}{2}}\,
\frac{(q;q)_d}{(q;q)_{d-l(\mu)}}
\sum_{l_1,\dots,l_n\geq 0}
K\big(d-l(\mu),l_1\big)\prod_{i=2}^n K\big(l_{i-1}+m_{i-1}(\mu),l_i\big).
\]
By Lemma~\ref{Kabsum} we can now carry out the remaining $n$ sums, so that
\[
S_N=q^{-\sum_{i=1}^n\binom{\mu'_i}{2}}\,
\frac{(q;q)_d}{(q;q)_{d-l(\mu)}}.
\]
Since this is independent of $N$,
\[
M_{p,\mu}\big(\mGartond\big)
=q^{-\sum_{i=1}^n\binom{\mu'_i}{2}}\,
\frac{(q;q)_d}{(q;q)_{d-l(\mu)}}.
\]
Taking the $d\to\infty$ limit, we finally get
\[
M_{p,\mu}\big(\mGarton\big)
=q^{-\sum_{i=1}^n\binom{\mu'_i}{2}}.\qedhere
\]
\end{proof}

\section{Technical lemmas}

\subsection{Lemmas for $K(a,b)$}\label{Subsec: Technical Lemmas K(a,b)}

In this section we will prove Lemmas~\ref{Kabsum}, \ref{Lem: Sum K(a,b) q^-kb} and \ref{minusonesum}.
To this end we prove a result that implies all three lemmas as easy corollaries.

\begin{prop}\label{Prop: sum K(a,b) t^b}
Let $a$ be a nonnegative integer.
Then
\begin{equation}\label{t-generating function}
\sum_{b=0}^a K(a,b) t^b
=\qhyp{2}{0}{t,q^{-a}}{\text{--}}{q,-q^{a+1}}.
\end{equation}
\end{prop}

We remark that the essence of this lemma is a change of basis in the polynomial ring $(\mathbb{Z}[q])[t]$ from the standard basis $\{t^b\}_{b\geq 0}$ to the basis $\{(t;q)_b\}_{b\geq 0}$.

\begin{proof}
According to \eqref{Kab},
\[
K(a,b)=q^{\binom{b+1}{2}} (q;q^2)_{\ceil{\frac{a-b}{2}}} \qbin{a}{b}_q.
\]
We now use \eqref{RS} to write $(q;q^2)_{\ceil{\frac{a-b}{2}}}$ as a sum.
Then
\begin{align*}
\sum_{b=0}^a K(a,b) t^b
&= \sum_{b=0}^a \qbin{a}{b}_q q^{\binom{b+1}{2}} t^b \sum_{k=0}^{a-b} 
\qbin{a-b}{k}_q (-q)^k \\
&= \sum_{k=0}^a \qbin{a}{k}_q (-q)^k \sum_{b=0}^{a-k} 
\qbin{a-k}{b}_q q^{\binom{b}{2}} (qt)^b \\
&= \sum_{k=0}^a \qbin{a}{k}_q (-q)^k (-qt;q)_{a-k} \\
&=(-qt;q)_a \, \qhyp{2}{1}{0,q^{-a}}{-q^{-a}/t}{q,\frac{q}{t}},
\end{align*}
where the second-last equality follows from the $q$-binomial theorem \eqref{qbt}.
Finally, \eqref{t-generating function} results by application of the transformation formula \eqref{III8} for $(n,c,z)=(a,-q^{-a}/t,q/t)$.
\end{proof}

\begin{cor}
Lemma~\ref{Kabsum} holds.
\end{cor}

\begin{proof}
Since for $t=1$ the ${_2\phi_0}$ series on the right-hand side of \eqref{t-generating function} trivializes to $1$, the claim follows.
\end{proof}

\begin{cor}
Lemma~\ref{Lem: Sum K(a,b) q^-kb} holds.
\end{cor}

\begin{proof}
Let $k$ be a nonnegative integer.
Specializing $t=q^{-k}$ in \eqref{t-generating function} then gives
\[
\sum_{b=0}^a K(a,b) q^{-bk}
=\qhyp{2}{0}{q^{-k},q^{-a}}{\text{--}}{q,-q^{a+1}}.
\]
We once again apply \eqref{III8}, this time with $(n,c,z)=(k,-1,q^{a+1})$,
to find
\[
\sum_{b=0}^a K(a,b) q^{-bk}
=q^{-\binom{k}{2}} (-1;q)_k \: \qhyp{2}{1}{0,q^{-k}}{-1}{q,q^{a+1}}.
\]
This is equivalent to the statement of Lemma~\ref{Lem: Sum K(a,b) q^-kb}.
\end{proof}

\begin{cor}
Lemma~\ref{minusonesum} holds.
\end{cor}

\begin{proof}
Let $k$ be a positive integer.
Specializing $t=-q^{-k}$ in \eqref{t-generating function} gives
\[
\sum_{b=0}^a (-1)^b K(a,b) q^{-bk}
=\qhyp{2}{0}{-q^{-k},q^{-a}}{\text{--}}{q,-q^{a+1}}.
\]
First we apply the transformation formula \eqref{III7} for $(n,b,z)=(a,-q^{-k},-q^{a+1})$ to obtain
\[
\sum_{b=0}^a (-1)^b K(a,b) q^{-bk}
=(-1)^a q^{a+1-\binom{k+1}{2}} \qhyp{3}{2}{-q^{-k},q^{1-k},q^{-a}}{0,0}{q,q}.
\]
Assuming $k$ is a positive integer, by \eqref{III6} for $(n,c,z)=(k-1,-q^2,q^{a+1})$ this can be further rewritten as
\[
\sum_{b=0}^a (-1)^b K(a,b) q^{-bk}
=(-1)^a q^{a+1-\binom{k+1}{2}} (-q^2;q)_{k-1}\:
\qhyp{2}{1}{0,q^{1-k}}{-q^2}{q,q^{a+1}}.
\]
Since this is the lemma expressed in basic hypergeometric form, we are done.
\end{proof}

\subsection{The shape of $\A_r$}\label{sec:shapeA_r}

\begin{lemma}\label{lem:Am-is-cube}
For $r\geq 1$, $\A_r$ is a $2^{r-1}$-dimensional cube. 
\end{lemma}

\begin{proof}
For each subset $T\subseteq [r]$, consider the linear form
\[
y_T(\ba)=\sum_{S\in \rodd} (-1)^{\abs{S\cap T}}\ba(S).
\]
Recall from the defining inequalities of $\A_r$ that $y_T(S)\leq 1$ for every subset $T\subseteq [r]$.
If we denote the complement of $T$ with respect to $[r]$ by $T^c=[r]\setminus T$, then $y_{T^c}(\ba)=-y_T(\ba)$.
Let $X=\{T\subseteq [r]:\ 1\notin T\}$.
Since exactly one of $T$ and $T^c$ is in $X$, the defining inequalities for $\A_r$ can be restated as $\abs{y_T(\ba)}\leq 1$ for every $T\in X$.
Now consider the linear map
\[
H:\mathbb{R}^{\rodd}\to \mathbb{R}^{X},\qquad (H\ba)(T)=y_T(\ba).
\]
Then $\A_r=H^{-1}([-1,1]^{2^{r-1}})$.
Hence we want to show that $H$ is a scalar multiple of an orthogonal map. We will show that $H^TH= 2^{r-1}I$.

We identify $X$ with $\rodd$ by defining
\[
S(U)=
\begin{cases}
U, & \text{if $\abs{U}$ is odd},\\
U\cup\{1\}, & \text{if $\abs{U}$ is even},
\end{cases}
\]
for $U\in X$.
Note that we have $T\cap S(U)=T\cap U$ for $T,U\in X$.
Therefore, for $T\in X$
\[
y_T(\ba)
=\sum_{U\in X}
(-1)^{\abs{T\cap U}}
\ba\big(S(U)\big).
\]
Our goal is thus to show that for $T_1,T_2\in X$,
\[
\sum_{U\in X} 
(-1)^{\abs{T_1\cap U}} 
(-1)^{\abs{T_2\cap U}}
=\begin{cases}
2^{r-1} &\text{if $T_1=T_2$},\\
0 &\text{otherwise}.
\end{cases}
\]
Let $T_1 \Delta T_2$ be the symmetric difference of $T_1$ and $T_2$.
Then
\[
(-1)^{\abs{T_1\cap U}} (-1)^{\abs{T_2\cap U}} 
=(-1)^{\abs{(T_1\Delta T_2) \cap U}}.
\]
The result now follows from the fact that
\[
\sum_{U\in X} (-1)^{\abs{(T_1\Delta T_2) \cap U }}
=\begin{cases}
2^{r-1} &\text{if $T_1\Delta T_2 =\emptyset$},\\
0 &\text{otherwise}.
\end{cases}\qedhere
\]
\end{proof}

\begin{cor}\label{Cor: H is injectiev}
Let $\ba,\bb\in\A_r$ such that
\[
\sum_{S\in\rodd}(-1)^{\abs{S\cap T}}\ba(S)
=\sum_{S\in \rodd}(-1)^{\abs{S\cap T}}\bb(S)
\]
for every subset $T\subseteq [r]$.
Then $\ba=\bb$.
\end{cor}

\begin{proof}
Let $H$ be the map defined in the proof of Lemma~\ref{lem:Am-is-cube}.
The conditions on $\ba$ and $\bb$ imply that $H(\ba)=H(\bb)$. Since $H$ was shown to be injective, this further implies that $\ba=\bb$.
\end{proof}

\section*{Acknowledgments}

The first author was supported by Simons Foundation Grant 917224.  The second author was supported by NSF grant DMS 2154223.  We thank Melanie Matchett Wood, Roger Van Peski, and Jiahe Shen for helpful comments.

\bibliographystyle{plain}
\bibliography{FKSW}

@article{nguyen2024universality,
  title={Universality for cokernels of random matrix products},
  author={Nguyen, Hoi H and Van Peski, Roger},
  journal={Advances in Mathematics},
  volume={438},
  pages={Paper No. 109451, 70},
  year={2024},
  publisher={Elsevier}
}

@article{wood2017distribution,
  title={The distribution of sandpile groups of random graphs},
  author={Wood, Melanie Matchett},
  journal={Journal of the American Mathematical Society},
  volume={30},
  number={4},
  pages={915--958},
  year={2017}
}

@book{macdonald1995symmetric,
  title={Symmetric {F}unctions and {H}all {P}olynomials},
  author={Macdonald, Ian G},
  year={1998},
  publisher={Oxford University Press}
}

@article{clancy2015cohen,
  title={On a {C}ohen--{L}enstra heuristic for {J}acobians of random graphs},
  author={Clancy, Julien and Kaplan, Nathan and Leake, Timothy and Payne, Sam and Wood, Melanie Matchett},
  journal={Journal of Algebraic Combinatorics},
  volume={42},
  number={3},
  pages={701--723},
  year={2015},
  publisher={Springer}
}

@article{warnaar2012dedekind,
  title={Dedekind's $\eta$-function and {R}ogers--{R}amanujan identities},
  author={Warnaar, S Ole and Zudilin, Wadim},
  journal={Bulletin of the London Mathematical Society},
  volume={44},
  number={1},
  pages={1--11},
  year={2012},
  publisher={Wiley Online Library}
}

@article{berkovich2005positivity,
  title={Positivity preserving transformations for $q$-binomial coefficients},
  author={Berkovich, Alexander and Warnaar, S Ole},
  journal={Transactions of the American Mathematical Society},
  volume={357},
  number={6},
  pages={2291--2351},
  year={2005}
}

@article{fulman2019random,
  title={Random partitions and {C}ohen--{L}enstra heuristics},
  author={Fulman, Jason and Kaplan, Nathan},
  journal={Annals of Combinatorics},
  volume={23},
  number={2},
  pages={295--315},
  year={2019},
  publisher={Springer}
}

@article{meszaros2020distribution,
  title={The distribution of sandpile groups of random regular graphs},
  author={M{\'e}sz{\'a}ros, Andr{\'a}s},
  journal={Transactions of the American Mathematical Society},
  volume={373},
  number={9},
  pages={6529--6594},
  year={2020}
}

@article{koplewitz2023sandpile,
  title={Sandpile groups of random bipartite graphs},
  author={Koplewitz, Shaked},
  journal={Annals of Combinatorics},
  volume={27},
  number={1},
  pages={1--18},
  year={2023},
  publisher={Springer}
}

@article{bhargava2023rank,
  title={The rank of the sandpile group of random directed bipartite graphs},
  author={Bhargava, Atal and DePascale, Jack and Koenig, Jake},
  journal={Annals of Combinatorics},
  volume={27},
  number={4},
  pages={979--992},
  year={2023},
  publisher={Springer}
}

@inproceedings{friedman1989distribution,
  title={On the distribution of divisor class groups of curves over a finite field},
  author={Friedman, Eduardo and Washington, Lawrence C},
  booktitle={Th{\'e}orie des nombres},
  pages={227--239},
  year={1989},
  organization={de Gruyter Berlin}
}

@article{malle2010distribution,
  title={On the distribution of class groups of number fields},
  author={Malle, Gunter},
  journal={Experimental Mathematics},
  volume={19},
  number={4},
  pages={465--474},
  year={2010},
  publisher={Taylor \& Francis}
}

@article{garton2015random,
  title={Random matrices, the {C}ohen--{L}enstra heuristics, and roots of unity},
  author={Garton, Derek},
  journal={Algebra \& Number Theory},
  volume={9},
  number={1},
  pages={149--171},
  year={2015},
  publisher={Mathematical Sciences Publishers}
}

@book{GR04,
  title={Basic {H}ypergeometric {S}eries},
  author={Gasper, George and Rahman, Mizan},
  series={Encyclopedia of Mathematics and its Applications},
  volume={96},
  edition={Second},
  year={2004},
  publisher={Cambridge University Press, Cambridge}
}

@article{bytev2021q,
  title={$q$-{D}erivatives of multivariable $q$-hypergeometric function with respect to their parameters},
  author={Bytev, Vladimir V and Zhang, Pengming},
  journal={Physics of Particles and Nuclei Letters},
  volume={18},
  number={3},
  pages={284--289},
  year={2021},
  publisher={Springer}
}

@misc{FulmanKaplanSinghalWarnaar_SandpileBipartiteRepo,
  author       = {Jason Fulman and Nathan Kaplan and Deepesh Singhal and S Ole Warnaar},
  title        = {Sandpile Groups of Random Bipartite Graphs -- Numerical Experiments},
  howpublished = {\url{https://github.com/DDeepuS/Sandpile_group_Bipartite_graph}},
  year         = {2026},
  note         = {GitHub repository},
}

@article{lipnowski2020cohen,
  title={Cohen--{L}enstra heuristics and bilinear pairings in the presence of roots of unity},
  author={Lipnowski, Michael and Sawin, Will and Tsimerman, Jacob},
  journal={\href{https://arxiv.org/abs/2007.12533}{arXiv:2007.12533}},
  year={2020}
}

@inproceedings{wood2022probability,
  title={Probability theory for random groups arising in number theory},
  author={Wood, Melanie Matchett},
  booktitle={Proceedings of the International Congress of Mathematicians},
  volume={6},
  pages={4476--4508},
  year={2022}
}

@article{poonen2012random,
  title={Random maximal isotropic subspaces and {S}elmer groups},
  author={Poonen, Bjorn and Rains, Eric},
  journal={Journal of the American Mathematical Society},
  volume={25},
  number={1},
  pages={245--269},
  year={2012}
}

@article {NguyenWood,
    AUTHOR = {Nguyen, Hoi H and Wood, Melanie Matchett},
     TITLE = {Local and global universality of random matrix cokernels},
   JOURNAL = {Math. Ann.},
  FJOURNAL = {Mathematische Annalen},
    VOLUME = {391},
      YEAR = {2025},
    NUMBER = {4},
     PAGES = {5117--5210},
      ISSN = {0025-5831,1432-1807},
   MRCLASS = {15B52 (60B20)},
  MRNUMBER = {4884544},
       DOI = {10.1007/s00208-024-03050-0},
       URL = {https://doi.org/10.1007/s00208-024-03050-0},
}

@article {Hodges,
    AUTHOR = {Hodges, Eliot},
     TITLE = {The distribution of sandpile groups of random graphs with
              their pairings},
   JOURNAL = {Trans. Amer. Math. Soc.},
  FJOURNAL = {Transactions of the American Mathematical Society},
    VOLUME = {377},
      YEAR = {2024},
    NUMBER = {12},
     PAGES = {8769--8815},
      ISSN = {0002-9947,1088-6850},
   MRCLASS = {05C80 (15B52 60B20)},
  MRNUMBER = {4815522},
MRREVIEWER = {Regino\ Criado},
       DOI = {10.1090/tran/9244},
       URL = {https://doi.org/10.1090/tran/9244},
}

@article{SawinWood,
  title={The moment problem for random objects in a category},
  author={Sawin, Will and Wood, Melanie Matchett},
  journal={\href{https://arxiv.org/abs/2210.06279}{arXiv:2210.06279}},
  year={2024}
}

@article{Singhal_sandpile,
  title={Sandpile groups of random bipartite graphs},
  author={Singhal, Deepesh},
  journal={Preprint},
  year={2026}
}

@article {vanPeski2025,
    AUTHOR = {Van Peski, Roger},
     TITLE = {Symmetric functions and the explicit moment problem for
              abelian groups},
   JOURNAL = {Proc. Amer. Math. Soc.},
  FJOURNAL = {Proceedings of the American Mathematical Society},
    VOLUME = {153},
      YEAR = {2025},
    NUMBER = {7},
     PAGES = {2799--2812},
      ISSN = {0002-9939,1088-6826},
   MRCLASS = {20K01 (05E05)},
  MRNUMBER = {4912116},
       DOI = {10.1090/proc/17236},
       URL = {https://doi.org/10.1090/proc/17236},
}

\end{document}